\documentclass[12pt]{article}
\usepackage{defs1}
\usepackage{amsmath}                    
\usepackage{amssymb}                                    
\usepackage{amsfonts}
\usepackage[normalem]{ulem}
\usepackage{xcolor}

\definecolor{mathgreen}{HTML}{009b77}
\usepackage{stmaryrd} 

\title{{Differential cohomology theories as sheaves of spectra}}
\author{Ulrich Bunke\thanks{Fakult\"at f\"ur Mathematik,
Universit{\"a}t Regensburg,
93040 Regensburg,
GERMANY, ulrich.bunke@mathematik.uni-regensburg.de} , Thomas Nikolaus\thanks{Fakult\"at f\"ur Mathematik,
Universit{\"a}t Regensburg,
93040 Regensburg,
GERMANY,  thomas.nikolaus@mathematik.uni-regensburg.de} and Michael V\"olkl\thanks{Fakult\"at f\"ur Mathematik,
Universit{\"a}t Regensburg,
93040 Regensburg,
GERMANY, michael.voelkl@mathematik.uni-regensburg.de}
}
\newcommand{\Fib}{{\tt Fib}}
\newcommand{\bH}{\mathbf{H}}
 \newcommand{\hol}{\mathrm{hol}}
\newcommand{\BCH}{\mathbf{BCH}}
\newcommand{\BCS}{\mathbf{\widetilde{BCH}}}

 \newcommand{\cR}{\mathcal{R}}
\newcommand{\Pure}{\Fun^{desc,pure}}

\newcommand{\cK}{\mathcal{K}}

\renewcommand{\sp}{\mathrm{sp}}
\newcommand{\Iso}{\mathrm{Iso}}

\newcommand{\CAlg}{{\mathbf{CAlg}}}
\newcommand{\bs}{\mathbf{s}}

\newcommand{\cZ}{{\mathcal{Z}}}

\newcommand{\const}{{\tt const}}

\newcommand{\cU}{{\mathcal{U}}}

 \newcommand{\Cone}{{\tt Cone}}
 \newcommand{\fib}{{\tt Fib}}
 
 \newcommand{\Vect}{{\tt Vect}}
 
 \newcommand{\CommMon}{{\mathbf{CMon}}}
 \newcommand{\Cat}{{\mathbf{Cat}}}

\newcommand{\Group}{\mathbf{Grp}}

\newcommand{\Mor}{{\tt Mor}}

\newcommand{\ku}{{\mathbf{ku}}}
\newcommand{\cE}{{\mathcal{E}}}
\newcommand{\tot}{\mathrm{tot}}
\newcommand{\Mfo}{\Mf^{op}}

\newcommand{\Comm}{\mathbf{C}}
\newcommand{\Grp}{\mathbf{Grp}}
\newcommand{\bku}{\mathbf{ku}}

\newcommand{\iso}{\xymatrix{\ar^\sim[r]&}}
\renewcommand{\leftrightarrows}{\xymatrix{ \ar@<0.4ex>[r] & \ar@<0.5ex>[l]}}

\newcommand{\Groupoids}{\mathbf{Grpd}}
\newcommand{\Ad}{\mathrm{Ad}}

\begin{document}
\maketitle
\begin{abstract}
We show that every sheaf on the site of smooth manifolds with values in a stable $(\infty,1)$-category (like spectra or chain complexes)
gives rise to a ``differential cohomology diagram'' and a homotopy formula, which  are common features of all classical examples of differential cohomology theories. 
These structures are naturally derived from a canonical decomposition of a sheaf into a homotopy invariant part and a piece which has a trivial evaluation on a point. In the classical examples the latter
is the contribution of differential forms. 
This decomposition suggest a natural scheme to analyse new sheaves
by determining these pieces and the {gluing} data.  We perform this analysis for a variety of classical
and not so classical examples.
 \end{abstract}
\tableofcontents

\section{Introduction}
A differential cohomology theory is a refinement of a cohomology theory in the sense  of homotopy theory to smooth manifolds.  {In particular a differential cohomology theory assigns to any smooth manifold a (graded) abelian group of differential cohomology classes}. In most examples  {these groups} combine homotopical information with local differential form data  {on the manifold}. 

In the present paper a differential cohomology theory is a sheaf on the category of manifolds with values  {in the category of chain complexes or the category of spectra. More precisely we consider these categories as higher categories and more generally allow sheaves with values in an arbitrary stable and presentable $(\infty,1)$-category $\bC$}.  Our goal is to analyse  these sheaves from a general point of view.  

\bigskip

{The main structures of a differential cohomology theory are a curvature map, the underlying class map, and a map delivering the homotopies between differential cohomology classes with the same underlying {class}. A common notation for these maps is $R$, $I$, and $a$. These}
 structures of a differential cohomology are nicely encoded in the differential cohomology diagram  \eqref{ghb1} {which was} popularized in this form by Simons-Sullivan \cite{MR2365651}.
 {A} differential cohomology  {theory} furthermore comes with 
  a homotopy formula \eqref{ghb3} which quantifies the fact that  it is not homotopy invariant  in general.
  
{These structures} are usually derived as consequences of the specific construction of the differential cohomology theory. Our main observation is that  {they are present}    in complete generality
for any $\bC$-valued sheaf.  We explain the details of  {these abstract constructions} in Section \ref{ttz13}.  

{
In the literature most examples of differential cohomology theories are given as functors, denoted like $\hat E^{k}$ or similar,  from smooth manifolds to abelian groups or as collections of such functors indexed by $\Z$ or $\Z/2\Z$ reflecting
a $\Z$ or $\Z/2\Z$-grading of an underlying generalized cohomology theory $E^{k}$.}
Axiomatizations for differential cohomology  {presented in this form were}  given in \cite{MR2365651}, \cite{MR2608479}.
First examples were constructed by cycles and relations, e.g.  the refinements of integral cohomology 
\cite{MR827262} or differential $K$-theory \cite{MR2664467}, see  {also} \cite{2010arXiv1011.6663B} for an overview and more references.

  {Given a sheaf of spectra $\hat E$ on the category of smooth manifolds one gets a differential cohomology  {theory} in the classical form by  defining $$\hat E^{k}(M):=\pi_{-k}(\hat E(M))\ .$$
 A general construction of differential refinements of  arbitrary generalized cohomology theories satisfying
 the axioms proposed in \cite{MR2608479} was given in \cite{MR2192936}. We will review this   Hopkins-Singer construction  in Example \ref{ttz12}. Note that the axioms \cite{MR2608479} prescribe the image of the curvature map as closed differential forms with coefficients in some graded vector space.
The target of the curvature map for a general differential cohomology theory 
given by an arbitrary sheaf of spectra $\hat E$ will not necessarily be related with differential forms, see the extreme Example \ref{uuzzii60}. In particular, the groups $\hat E^{k}$ will not satisfy the axioms  as stated  {in} \cite{MR2608479}.     }

The main technical tool for our analysis of  {sheaves on} the category of manifolds are the left and right-adjoints of the functor which maps 
an object of $\bC$ to a constant sheaf. While  the right-adjoint exists  {for all categories of sheaves}, the left-adjoint, called homotopification,  is a  speciality of the category of smooth manifolds. 
The fact that  manifolds are locally contractible implies that the constant sheaf functor identifies $\bC$ with the full subcategory of homotopy invariant sheaves. We will explain all these functors in detail in Section  \ref{ttz14}. 

 {A related (non-stable) setup} of a category of sheaves with these {adjoint} functors {was}  axiomatized as a cohesive $(\infty,1)$-topos in the work of Urs Schreiber, {see also Remark \ref{rtrtrt}. 
}

 {In this paper we} use these adjoints in order to deconstruct a given sheaf into its underlying homotopy invariant 
part, cycle data  {and a characteristic map which contains the information how the homotopy invariant part and the cycle data is glued together, see Definition \ref{uuzzii70}.}  Proposition  \ref{wie4} suggests a natural way for the analysis of a differential cohomology theory by calculating its underlying homotopy invariant part,
its cycle data, and the characteristic map.

\bigskip
 
{The remainder of the paper consists of a collection of examples of sheaves of chain complexes or spectra obtained by natural geometric constructions which we analyse according to the lines indicated above.
}

In Section \ref{cl12} we perform this analysis for  the differential cohomology theories given by the classical (Hopkins-Singer type) construction.   In Section \ref{ttz30} we analyse sheaves of spectra represented by abelian Lie groups and show that they are equivalent to sheaves obtained from the classical construction. {Further we} consider the universal differential cohomology theories receiving characteristic classes for $G$-principal bundles with and without connection. We explain {how} the  Cheeger-Simons construction of differential refinements of characteristic classes for $G$-bundles with connection can be formulated in the language developed in the present paper.

Finally, in Section \ref{ttz15} we introduce a version of differential $K$-theory which is  universal
for additive characteristic classes of complex vector bundles with or without connections.
A complete understanding of its data is an open problem which might be of some interest. 
We discuss the relation of this universal differential $K$-theory with the usual Hopkins-Singer type differential $K$-theory and the more recent loop differential $K$-theory introduced in \cite{2012arXiv1201.4593T}.
Moreover we give a Snaith-type construction of a differential refinement of periodic complex $K$-theory.

Most of the calculations are direct consequences of the calculation of the left and right-adjoint of the constant sheaf functor applied to truncated twisted de Rham complexes or sheaves obtained by the Yoneda embedding. We collect these and some other technical results in Section \ref{gfd1}.

\bigskip

{\em Acknowledgement: The authors have benefited from various discussion with Urs Schreiber whose approach to differential cohomology via cohesive $\infty$-topoi  has a notable overlap with our work.
A great portion of the basic ideas used in the present papers grew out of our  cooperation with David Gepner.  We would also like to thank Peter Teichner and David Carchedi for helpful comments.
U.B and M.V thank the organizers of the GAP 2013 (Pittsburgh) conference for hospitality and providing the motivation for writing this note.}

\section{Sheaves on manifolds}\label{ttz14}

In the present paper we freely use the {theory and} language of $(\infty,1)$-categories. We try to work as model independent as possible but for concreteness we use quasi-categories. We mostly adopt the notational conventions of \cite{HTT} and  \cite{HA} and provide more specific references where necessary. 
 
We consider  an ordinary category ${\bC}$ as an $(\infty,1)$-category by taking the nerve. In order to simplify notation we also write $\bC$ for the $(\infty,1)$-category  and hope that this does not lead to confusion.

We let $\Mf$ denote the category of smooth manifolds with corners.  A typical object of $\Mf$ is the standard $n$-simplex $\Delta^{n}$, $n\in \nat$. In general, an $n$-dimensional manifold with corners  is locally modeled on the quadrant $[0,\infty)^{n}\subset \R^{n}$. An $\R$-valued function defined on an open subset  $U\subseteq [0,\infty)^{n}$ is smooth, if there exists an open subset $\tilde U\subseteq \R^{n}$ such that $U=\tilde U\cap [0,\infty)^{n}$ and a smooth extension of this function to $\tilde U$.  A map into a manifold with corners is smooth if its composition with the charts considered as maps to $\R^{n}$ are smooth. In particular, for a smooth map between manifolds with corners there are no restrictions of the kind that faces have to map to faces.
 
\begin{ddd} 
{Let $\bC$ be a general  $(\infty,1)$-category.}
The $(\infty,1)$-category  of presheaves on manifolds with values in $\bC$ is the    
 $(\infty,1)$-category  of functors $\Fun(\Mf^{op},\bC)$.
  \end{ddd}

 \begin{ex}{\rm
Every object $C\in \bC$ gives rise to a constant presheaf $\underline{C}\in \Fun(\Mf^{op},\bC)$
which associates to every manifold the object $C$, and to every smooth map the identity map.
In fact, we have an adjunction
 $$\underline{(-)}:\bC\leftrightarrows \Fun(\Mf^{op},\bC):\ev_{*}\ ,$$
 where $\ev_{*}$ is the evaluation at the final object in  $\Mf$, the point $*=\Delta^{0}$.
}\end{ex}

 The category $\Mf$ has a Grothendieck topology  which is determined by the collection of coverings which are  surjective submersions with discrete fibres. To a covering  $U\to M$ we can associate   the simplicial   smooth manifold (also called the \v{C}ech nerve of the covering) which will be denoted by
$U^{\bullet}\in \Fun(\Delta^{op},\Mf)$. Its evaluation at $[n]\in \Delta^{op}$ is given by   $$U^{\bullet}([n]):=\underbrace{U\times_{M}\dots\times_{M} U}_{n\times}\ ,\quad n\in \nat\ .$$ In particular, for $i\in \{0,\dots n\}$ the $i$'th face map
$\partial_{i}^{*}:U^{\bullet}([n+1])\to U^{\bullet}([n])$  is  the obvious projection leaving out
the $i$-th factor.
If we evaluate a presheaf 
 $F\in \Fun(\Mf^{op},\bC)$ on the simplicial manifold $U^{\bullet}$, then  we get a cosimplicial object $F(U^{\bullet})\in \Fun(\Delta,\bC)$ of $\bC$.

 \bigskip
 From now on we assume that $\bC$ is a (locally) presentable $(\infty,1)$-category. This in particular implies that $\bC$ has all colimits and limits. See \cite[Chapter 5]{HTT} for details. The category of presheaves with values in $\bC$ is then also presentable since the category $\Mf$ is essentially small.
 
\begin{ddd}
We say that $F\in \Fun(\Mf^{op},\bC)$ is a sheaf, if for any manifold $M$ and covering $U\to M$ the canonical map
$$F(M)\to \lim_{\Delta} F(U^{\bullet})$$
is an equivalence. We  denote the full subcategory of sheaves by $$\Fun^{desc}(\Mf^{op},\bC)\subseteq \Fun(\Mf^{op},\bC)\ .$$ 
\end{ddd}

 It has been shown in \cite[6.2.2.7]{HTT} that the category of sheaves is also presentable and that 
there exists an adjunction
$$L: \Fun(\Mf^{op},\bC)\leftrightarrows  \Fun^{desc}(\Mf^{op},\bC):\mathrm{inclusion}\ .$$
The left-adjoint $L$ is called \emph{sheafification}.

\begin{ddd}
A presheaf $F\in \Fun(\Mf^{op},\bC)$ is called homotopy invariant, if for all manifolds $M\in \Mf$ the canonical map
$F(M)\to F(\Delta^{1}\times  M)$ induced by the projection $\Delta^{1}\times M\to M$
is an equivalence. We let $$ \Fun^{h}(\Mf^{op},\bC)\subseteq \Fun(\Mf^{op},\bC)$$
be the full subcategory of homotopy invariant presheaves.
\end{ddd}

Using standard techniques one checks that the inclusion is an accessible reflective subcategory. Therefore the $(\infty,1)$-category of homotopy invariant presheaves is also presentable.
The inclusion of homotopy invariant presheaves into all presheaves not only  preserves limits but also colimits and therefore has both adjoints. In particular
 there exists an adjunction
\begin{equation}\label{hdef123}\cH^{ {\text{pre}}}: \Fun(\Mf^{op},\bC)\leftrightarrows  \Fun^{h}(\Mf^{op},\bC):\text{inclusion}\ .\end{equation}
 The left-adjoint $\cH^{ {\text{pre}}}$   is called homotopification.  
\begin{ddd}
We let $\Fun^{desc,h}(\Mf^{op},\bC)$ be the full subcategory of sheaves which are   homotopy invariant.
\end{ddd}
Now we list some properties which are crucial for the rest of the paper. 
\begin{prop}\label{propo}$\ $ 
\begin{enumerate}
\item
The functor $\const:=L\circ (\underline{-})$
induces  an equivalence  $$\const: \bC\iso \Fun^{desc,h}(\Mf^{op},\bC)$$
with inverse given by evaluation $\ev_*$ at $*\in \Mf$. 
\item
For a homotopy invariant presheaf  $F$ the sheafification $L(F)$ is also homotopy invariant. 
\item
Setting $\cH := L \circ \cH^{\text{pre}}\mid_{\Fun^{desc}(\Mf^{op},\bC)}$ we obtain an adjunction
 $$\cH: \Fun^{desc}(\Mf^{op},\bC)\leftrightarrows  \Fun^{desc,h}(\Mf^{op},\bC):\mathrm{inclusion}\ .$$
In general, the functor $\cH$, also called homotopification, does  not admit a further left adjoint. However in good cases, e.g. for $\bC$ stable or for $\bC$ the $(\infty,1)$-category of spaces, the functor $\cH$ preserves finite {products}. 
\item
Homotopification commutes with sheafification, i.e. {we have an equivalence} $L \circ\cH^{pre} \simeq \cH \circ L$.
\end{enumerate}
\end{prop}
\begin{proof}
The first statement is basically due to Dugger \cite{dugger} who shows it in a slightly different setting and only for $\bC$ the $(\infty,1)$-category of spaces. {The idea is as follows.} Let $F$ be a homotopy invariant sheaf and $M$ be a manifold.  {Then using a good open cover of $M$ and descent and homotopy invariance of  $F$  one obtains  equivalences  $$F(M) \simeq F(*)^{M_{top}}\simeq \const(F)(M)\ . $$   For more details see \cite[Sec. 6.5]{2013arXiv1306.0247B}.} 

For the second claim we use a similar argument. Assume  that $F$ is  a homotopy invariant  presheaf. We see from  \cite[6.2.2.7]{HTT} that we can write the evaluation $L(F)(M)$  of the sheafification of $F$ at a manifold $M$  as an iterated  colimit over good open covers of limits of the form of the form
$$ {{\lim}_{\Delta}} F(U^\bullet)\ .$$ In
 view of    the equivalence ${{\lim}_{\Delta}} F(U^\bullet) \simeq F(*)^{M_{top}}$  the diagram is constant.
 
  The last two statements are completely formal, see e.g. \cite[Problem. 4.33]{skript} for more details about $4.$
  \end{proof}

The functor $\const$ is the left-adjoint of an adjunction
$$\const:\bC\leftrightarrows   \Fun^{desc}(\Mf^{op},\bC):\ev_{*} .$$ 
   Consequently, we have an adjunction
  $$\mathrm{inclusion}:\Fun^{desc,h}(\Mf^{op},\bC)\leftrightarrows \Fun^{desc}(\Mf^{op},\bC):{\cS}\ ,$$
   where 
   \begin{equation}\label{vm1} {\cS:=\const\circ \ev_{*}}\ .\end{equation}
The functor $\cS$ is also the left-adjoint in an adjunction
$$\cS: \Fun^{desc}(\Mf^{op},\bC)\leftrightarrows  \Fun^{desc,h}(\Mf^{op},\bC):\mathcal{G}\ ,$$
where $\cG$ is the Godement functor given on objects by
\begin{equation}\label{mikoe}\cG(\hat E)(M):=  \hat E(*)^{M^{\delta}}\ ,\end{equation}
where $M^{\delta}$  is the topological space obtained by endowing the manifold  $M$ with the discrete topology.

\begin{rem}\label{rtrtrt}{\rm 
In summary, we have a quadruple adjunction
$$ 
(\cH \dashv \const \dashv \cS \dashv \cG ):
\Fun^{desc}(\Mf^{op},\bC) \mathrel{\substack{\stackrel{\cH}{\textstyle\rightarrow}\\
\stackrel{\const}{\textstyle\hookleftarrow}\\
\stackrel{\cS}{\textstyle\rightarrow} \\
\stackrel{\cG}{\textstyle\hookleftarrow}}} 
\Fun^{desc,h}(\Mf^{op},\bC) $$
where the two functors $\const$ and $\cG$ are full embeddings. In addition,  if $\bC$ is stable or the category of spaces $\sSet[W^{-1}]$, then the functor $\cH$ preserves finite products.
For $\bC$ the $(\infty,1)$-category of spaces the existence of  {such an adjoint quadruple {$(\cH,\const,\cS,\cG)$}} is part of the axiomatics of a cohesive topos as introduced by Schreiber \cite[Remark 3.4.2.]{ddddddd}. }\end{rem}

For later reference we fix  some notation.
\begin{ddd}\label{ijoisojsvs}
We consider the following functors $\Fun^{desc}(\Mf^{op},\bC)\to \bC$
$$U:=\ev_{*}\circ \cH\ , \qquad S:=\ev_{*}\ .$$
\end{ddd}

\section{Structures in the stable case}\label{ttz13}

This section presents the main general result of this note which essentially  consists in  a straightforward construction using the functors introduced in the previous section and taking functorial (co)fibres.
 From now on we assume that $\bC$ is a presentable and stable $(\infty,1)$-category. Stability implies in particular that the homotopy category $\mathrm{Ho}(\bC)$ is canonically triangulated. Typical examples are the $(\infty,1)$-categories of spectra or of chain complexes. For details we refer the reader to \cite[Chapter 1]{HA}.
We will use stability for the functorial construction and properties of fibre and cofibre sequences.
\begin{ddd}\label{ghb5}
We define the functors
  $\cA,\cZ:\Fun^{desc}(\Mf^{op},\bC)\to \Fun^{desc}(\Mf^{op},\bC)$
 to fit into the fibre sequences
 $$\cA\stackrel{a}{\to} \id\stackrel{I}{\to} \cH\to\Sigma \cA \qquad \text{and} \qquad \Sigma^{{-1}} \cZ\to \cS\to \id\stackrel{R}{\to} \cZ$$
using the unit and the counit of the adjunctions  {discussed in the last section}.
We furthermore define {a} functor $Z:  \Fun^{desc}(\Mf^{op},\bC)\to \bC$ by
$$Z:=\ev_{*}\circ\cH\circ\cZ\ \ . $$
\end{ddd} 

We consider a $\bC$-valued sheaf
 $\hat E\in \Fun^{desc}(\Mf^{op},\bC)$. Combining the  functors introduced  above and in Definition \ref{ijoisojsvs}   we see that it naturally fits into a diagram of
  vertical and horizontal fibre sequences
\begin{equation}\label{heut300010}
\xymatrix{
\Sigma^{-1}\const(Z(\hat E))\ar[r]\ar[d]&\cA(\hat E)\ar[r]^{d}\ar[d]^{a}&\cZ(\hat E)\ar@{=}[d]\\\ \const(S(\hat E))\ar[d]\ar[r]&\hat E\ar[d]^{I}\ar[r]^-{R}&\cZ(\hat E)\\
\const(U(\hat E))\ar[r]^-{{\simeq}}& {\cH(\hat E)}&}
\end{equation}
 such that the upper left square is cartesian.    By construction we have 
\begin{equation}\label{vk20}\cH(\cA(\hat E))\simeq 0 \quad \text{and} \quad \ev_{*}(\cZ(\hat E))\simeq 0  \ .\end{equation}

The symbols denoting these functors and maps are chosen to resemble the corresponding objects in standard examples of differential cohomology.  We have the following interpretations:
\begin{enumerate}
\item $U$ takes the underlying cohomology theory and $I$ maps a class to its underlying cohomology class.
\item $\cZ$ represents the differential cycles (often related with closed differential forms) and $R$ is the curvature map.
\item $S$ represents the  secondary cohomology theory (often given by the flat classes).
\item $\cA$ classifies the group of differential deformations (often differential forms up to the image of the de Rham differential).
\end{enumerate}
This will become clearer in the examples discussed in Section \ref{cl12}, see in particular \eqref{ghb100}.
If we let $\bC$ be the category of spectra or chain complexes, evaluate the sheaves on a smooth manifold $M$ and take the homotopy group $\pi_{-m}$ or the cohomology $H^{m}$, respectively, then we get the differential cohomology diagram
\begin{equation}\label{ghb1}\xymatrix{&{\cA(\hat E)^{m}(M)}\ar[rr]^{d}\ar[dr]^{a}&&{\cZ(\hat E)^{m}(M)}\ar[dr]&\\
H^{m-1}(M;Z(\hat E))\ar[ur]\ar[dr]&&{\hat E^{m}(M)}\ar[ur]^{R}\ar[dr]^{I}&&H^{m}(M;Z{(\hat E)})\\
&H^{m}(M;S(\hat E))\ar[ur]\ar[rr]&&H^{m}(M;U(\hat E)) \ar[ur]}\end{equation}
where the upper and the lower parts are pieces of long exact sequences, and the
two middle diagonal sequences are exact, too.
Here we use the notation $H^{m}(M;E):=\const(E)^{m}(M)$ for the application of the cohomology theory represented by $E$ to the manifold $M$ since the usual notation $E^{m}(M)$ could be confused with
the notation used for sheaves.

\bigskip
 
Now we want to derive an additional piece of structure that is present in differential cohomology theories, namely the homotopy formula. 
In order to do this we
 define the endofunctor $$I:~\Fun^{desc}(\Mf^{op},\bC)\to \Fun^{desc}(\Mf^{op},\bC)$$ such that
$$(IF)(M):=F(\Delta^{1}\times M)\ .$$
This is a special case of a more general functor which we introduce in \eqref{heut1000} and which uses a slightly different notation.

\begin{theorem}[Homotopy formula]\label{ghb2}
There exists a natural transformation (called the integration map)
\begin{equation}\label{heut1001}\int: I\cZ(\hat E)\to {\cA(\hat E)}\end{equation} such that we have  the homotopy formula
\begin{equation}\label{ghb3}\partial_{1}^{*}-\partial_{0}^{*} \simeq a\circ\int\circ R\end{equation}
as transformations $I\to \id$.
Moreover we have an equivalence $$\int\circ\pr^*\simeq0 \ .$$
\end{theorem}
\proof
At first we construct a map $\alpha:I\cZ(\hat E)\to \hat E$ by considering the following diagram:
\begin{equation}\label{kp21}
\xymatrixcolsep{5pc}
\xymatrix{
 & I\cS(\hat E)\ar[dd]_{\id}\ar[r]\ar[dl]_{\partial_0^*}^{\simeq}\ar@/^2pc/[ddd]^{0}
 & I\hat E\ar[ddd]^{\partial_1^*-\partial_0^*}\ar[r]^{R} & I\cZ(\hat E)\ar@{.>}[dddl]^{\alpha}\\
 \cS(\hat E)\ar[dr]^{\pr^*}_{\simeq}\ar@/_2pc/[ddr]^{0} & & &\\
 & I\cS(\hat E)\ar[d]_{\partial_1^*-\partial_0^*} & &\\
 & \cS(\hat E) \ar[r] & \hat E &\\
  \save[]+<2.1cm,4.4cm>*\txt<8pc>{\textcircled{1}}\restore
   \save[]+<1.8cm,2.6cm>*\txt<8pc>{\textcircled{2}}\restore
   \save[]+<3.8cm,3.6cm>*\txt<8pc>{\textcircled{3}}\restore
} \end{equation}
The left upper triangle \textcircled{1} is filled by a preferred homotopy
because the sheaf {$\cS(\hat E)$} is homotopy-invariant.
Since $I\cS(\hat E)$ is a functor the lower left triangle \textcircled{2} commutes (again up to a preferred homotopy).  The composition of these two homotopies yields the homotopy of the middle triangle \textcircled{3} . Finally, the middle pentagon commutes by the naturality of the counit $\cS(-) \to \id(-)$. It follows that the middle square commutes by a preferred homotopy. 
 The map $\alpha$ is now induced by the universal property of the cofibre. 
\\
Below we will use the notation $\alpha_{\hat E}$ in order to indicate the natural dependence of $\alpha$ on $\hat E$.   In general, 
if a sheaf $\hat F$ is homotopy invariant, then we have a natural equivalence $R\simeq 0$ and therefore $\alpha_{\hat F}\simeq 0$ by composition. Applying this to the homotopy invariant sheaf  $\cH(\hat E)$ we get $\alpha_{\cH(\hat E)}\simeq 0$ and 
  can define a lift of $\alpha_{\hat E}$ to the integration map $\int:I\cZ(\hat E)\to \cA(\hat E)$:
\begin{equation}\label{heut2000}
\xymatrixcolsep{5pc}
\xymatrix{
 & I\cZ(\hat E)\ar[r]^I\ar[d]^{\alpha_{\hat E}}\ar@{.>}[dl]_{\int} 
 & I\cZ(\cH(\hat E))\ar[d]^{\alpha_{\cH(\hat E)}}\\
 \cA(\hat E)\ar[r]^{a}& \hat E\ar[r]^{I}& \cH(\hat E)
}\ ,\end{equation}
again using  the universal property of the fibre of $I$.
\\
The equivalence  $$ a\circ\int\circ R \simeq \alpha_{{\hat E}} \circ R \simeq \partial_{1}^{*}-\partial_{0}^{*}$$ now follows immediately
  from the constructions of $\alpha$ and $\int$. 
\\
Finally we prove that $\int \circ \pr^* \simeq 0$. To this end we first consider the following diagram:
$$
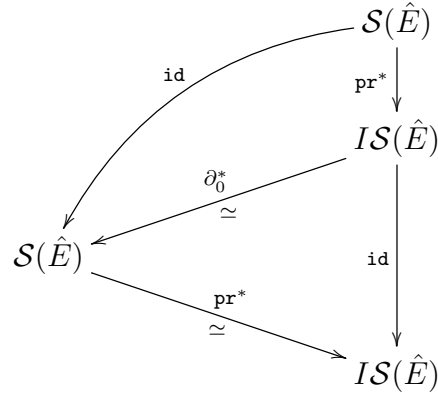
\xymatrixcolsep{8pc}
\xymatrix{
  & \cS(\hat E)\ar[d]_{\pr^*}\ar@/_2pc/[ddl]_{\id} \\ 
  & I\cS(\hat E)\ar[dd]_{\id}\ar[dl]_{\partial_0^*}^{\simeq}   \\
  \cS(\hat E)\ar[dr]^{\pr^*}_{\simeq} & \\
  & I\cS(\hat E)
}$$
The upper triangle is filled by functoriality and the lower one by homotopy invariance of the sheaf {$\cS(\hat E)$}. The outer triangle has two fillers: One is the composition of the two inner fillers described before and the other is induced from the left and right unitors. But by using homotopy invariance for a second time one sees that these fillers are homotopic.

Next we extend diagram \eqref{kp21}:
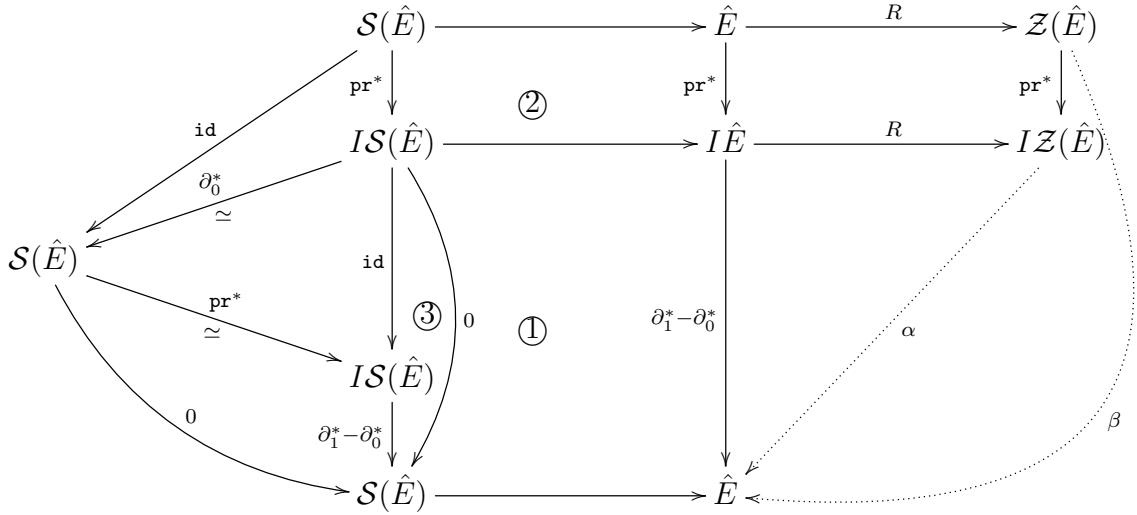
$$
\xymatrixcolsep{8pc}
\xymatrix{
  &\cS(\hat E)\ar[d]_{\pr^*}\ar[r]\ar[ddl]_{\id}
  & \hat E\ar[d]_{\pr^*}\ar[r]^{R}
  & \cZ(\hat E)\ar[d]_{\pr^*}\ar@{.>}@/^8pc/[ddddl]^{\beta}\\
  & I\cS(\hat E)\ar[dd]_{\id}\ar[r]\ar[dl]_{\partial_0^*}^{\simeq} \ar@/^2pc/[ddd]^{0}
  & I\hat E\ar[ddd]_{\partial_1^*-\partial_0^*}\ar[r]^{R} & I\cZ(\hat E)\ar@{.>}[dddl]^{\alpha}\\
  \cS(\hat E)\ar[dr]^{\pr^*}_{\simeq}\ar@/_2pc/[ddr]^{0} & & &\\
  & I\cS(\hat E)\ar[d]_{\partial_1^*-\partial_0^*} & &\\
  & \cS(\hat E) \ar[r] & \hat E &\\
  \save[]+<6.5cm,3.5cm>*\txt<8pc>{\textcircled{1}}\restore
  \save[]+<6.5cm,6.5cm>*\txt<8pc>{\textcircled{2}}\restore
  \save[]+<5.1cm,3.7cm>*\txt<8pc>{\textcircled{3}}\restore
  }$$
Again, the map $\beta$ is constructed by the universal property of the cofibre using the fillers of the two middle squares \textcircled1 and \textcircled2 which are induced by functoriality. Then clearly $\beta\simeq \alpha\circ \pr^{*}$.
\\
In the construction of $\alpha$ we used the filler of the inner triangle \textcircled3.
Now one checks using the above remarks that the composition of this filler with $\pr^*$ is homotopic to the filler of the outer left triangle induced by functoriality.
\\
So, up to homotopy, we can construct $\beta$ using the inner square (consisting out of \textcircled1, \textcircled2 and \textcircled3). But here both vertical compositions are homotopic to the trivial maps by functoriality. Hence we get $\beta\simeq 0$. {\hB}

\begin{rem}{\rm 
We will see in
Example \ref{gjl123} that the morphism $\int$ is related to {the} integration of differential forms over the interval $\Delta^{1}$. This example motivates the notation.}
\end{rem}

We now discuss a canonical decomposition of a sheaf   into its homotopy invariant part and cycle data. In this way we obtain  a sort of classification of differential cohomology theories.
 \begin{ddd} 
 We call $\hat E\in \Fun^{desc}(\Mf,\bC)$ \emph{pure} if $\hat E(*)\simeq 0$. We let $$\Pure(\Mf^{op},\bC)\subset \Fun^{desc}(\Mf^{op},\bC)$$ be the full subcategory of pure sheaves.
\end{ddd}

For example, the sheaf $\cZ(\hat E)$ of cycles of $\hat E$ (Definition \ref{ghb5}) is pure for any sheaf $\hat E$.
We moreover have an {adjunction}
$$\cZ:\Fun^{desc}(\Mf^{op},\bC)\leftrightarrows \Pure(\Mfo,\bC):{\mathrm{inclusion}}$$
which exhibits the functor $\cZ$ as the universal way to turn a sheaf into a pure sheaf.

\begin{prop}\label{wie4}
For any sheaf $\hat E \in \Fun^{desc}(\Mf^{op},\bC)$ the canonical diagram
  \begin{equation}\label{tsep22}\xymatrix{\hat E\ar[rr]^R\ar[d]&&\cZ(\hat E)\ar[d]\\
\cH(\hat E)\ar[rr]^-{\cH(R)}&&\cH(\cZ(\hat E))} \end{equation} 
is a pullback diagram. We moreover have a pullback square of $(\infty,1)$-categories:
$$\xymatrix{ \Fun^{desc}(\Mf^{op},\bC)\ar[r]^-{\cZ}\ar[d]^{U\to Z}&\Pure(\Mfo,\bC)\ar[d]^{U}\\
\Mor(\bC)\ar[r]^{{target}}&\bC}
$$
\end{prop}
\begin{proof}
For the first assertion we note that the fibre of $R: \hat E \to \cZ(\hat E)$ is by definition given by $\cS(\hat E)$. Since the functor $\cH$ preserves fibre sequences the 
fibre of $\cH(R): \cH(\hat E) \to \cH(\cZ(\hat E))$ is given by $\cH(\cS(\hat E))$. But since $\cS(\hat E)$ is already homotopy invariant, we have an equivalence of the fibres. This implies that the diagram is a pullback.

\medskip 

To prove the second claim we first abbreviate $\cE := \Fun^{desc}(\Mf^{op},\bC)^{\Delta[1] \times \Delta[1]}$ for the category of commuting squares in sheaves. We let $\cE_0$ be the full subcategory of those squares which have the property:
\begin{itemize}
\item the square is a pullback
\item the upper right object is pure
\item the lower objects are homotopy invariant.
\item the right vertical morphism is homotopification
\end{itemize}
Then we have an equivalence $\cE_0 \simeq \Mor(\bC) \times_{\bC} \Pure(\Mfo,\bC)$ since that category of homotopy invariant sheaves is equivalent to $\bC$ and since 
the squares are pullback squares.
Therefore we are left to show that $\cE_0$ is equivalent to $\Fun^{desc}(\Mf^{op},\bC)$.
This can be seen by noting that there are functors both ways: the functor 
$\eta: \cE_0 \to \Fun^{desc}(\Mf^{op},\bC)$ given by evaluating at the upper left corner and a functor
$\Phi: \Fun^{desc}(\Mf^{op},\bC) \to \cE_0$ 
which is defined as follows. First we assemble the natural transformations $R: \id \to \cZ$ and $I: \id \to \cH$ into a functor
$$\Fun^{desc}(\Mf^{op},\bC) \to \cE.$$

The first part of this lemma then implies that this functor factors through $\cE_0 \subset \cE$ and thus defines $\Phi$. 
We clearly have $\eta \circ \Phi {\simeq} \id$. 

Thus it remains to show that $\Phi \circ \eta \simeq \id$. For a pullback square of the form
$$\xymatrix{
\hat E \ar[r]\ar[d] & P \ar[d]\\
H \ar[r] & Q
}$$
in $\cE_0$ we can conclude that $\cZ(\hat E) \simeq P$ and $\cH(\hat E) \simeq H$ and 
$\cH (\cZ (\hat E)) \simeq Q$ using the fact that all three functors preserve pullbacks. The naturality of these equivalences follows from the fact that they come from the respective universal properties and thus can be assembled together in the required natural transformation. 
\end{proof}

\begin{ddd}
A differential refinement of $E\in \bC$ is a sheaf $\hat E\in  \Fun^{desc}(\Mf^{op},\bC)$
together with an equivalence $U(\hat E)\simeq E$.
 \end{ddd}

Therefore in order to determine a differential refinement $\hat E$ of $E$ we must choose a pure sheaf $\cZ{\in \Pure(\bC)}$ and a morphism ${\phi}:E\to U(\cZ)$.
\begin{ddd} \label{uuzzii70}We call $\cZ$ the cycle data and $\phi$ the characteristic map of the differential refinement $\hat E$ of $E$.
\end{ddd}
As we see from Proposition \ref{wie4} for a 
given cycle data and characteristic map, we can construct $\hat E$ by the pull-back
  \begin{equation}\label{heute1}\xymatrix{\hat E\ar[rr]\ar[d]&&\cZ\ar[d]\\
\const(E)\ar[rr]^-{\const(\phi)}&&\const( U(\cZ))}\ . \end{equation}
Furthermore, we have equivalences
$$\cA(\hat E)\simeq \cA(\cZ)\ , \quad  \cZ(\hat E)\simeq \cZ $$ 
and $S(\hat E)$ fits into the fibre sequence
$$S(\hat E)\to E\stackrel{\phi}{\to} U(\cZ)\to \Sigma S(\hat E)\ . $$

Let $\hat F\in \Fun^{desc}(\Mf^{op},\bC)$ be a differential refinement of $F\in \bC$ and $c:E\to F$ be a morphism in $\bC$. Then we can construct a sheaf $\hat E\in \Fun^{desc}(\Mf^{op},\bC)$ by forming the pull-back
$$\xymatrix{ 
   \hat E\ar[rr] \ar[d] && \hat F\ar[d]\\
   \const(E)\ar[rr]^-{\const(c)}&&\const(F)} \ .
$$
Since many of  our  examples will be constructed in this way let us express the main data of $\hat E$ in terms of the data of $\hat F$. The following Lemma will be used repeatedly without further notice  in calculations.
\begin{lem}\label{heut1003}$\ $ 
\begin{enumerate}
\item $U(\hat E)\simeq E$. In particular $\hat E$ is a differential refinement of $E$.  
\item $\cA(\hat E)\simeq \cA(\hat F)$.
\item $\cZ(\hat E)\simeq \cZ(\hat F)$.
\item $Z(\hat E)\simeq Z(\hat F)$
\item $S(\hat E)$ is given by the pull-back $$\xymatrix{S(\hat E)\ar[r]\ar[d]&S(\hat F)\ar[d]\\E\ar[r]^{c}&F}\ .$$
\item The characteristic map $\phi_{\hat E}$ is the composition  $E\stackrel{c}{\longrightarrow} F\stackrel{\phi_{\hat F}}{\longrightarrow} Z(\hat F)$ of $c$ and the characteristic map $\phi_{\hat F}$ of $\hat F$.
\item Under the equivalences $2.$ and $3.$  the integration maps \eqref{heut1001} for $\hat E$ and $\hat F$ coincide.
\end{enumerate}
\end{lem}
\proof All these properties follow from the definitions in a straightforward manner. \hB  

\section{Examples related with differential forms}\label{cl12}

In the following we present a series of examples of sheaves which we analyse by calculating the invariants introduced in Section  \ref{ttz13}. 
In this section the target category $\bC$ will be the category of chain complexes {in most examples.} More precisely, we denote by $\Ch$ the category of chain complexes of abelian groups and take $\bC:=\Ch[W^{-1}]$, the category obtained from $\Ch$  by formally inverting the quasi-isomorphisms. This category is a stable $(\infty,1)$-category (shown e.g. in \cite[Chapter 1]{HA}). The canonical localization functor is denoted by $\iota:\Ch\to \Ch[W^{-1}]$.

The calculations in this section mainly rely on  Lemma \ref{heut1003} and the Lemmas \ref{uuzzii1} and \ref{rrr4} which we prove in an appendix. 
Using these lemmas the proofs are straightforward, we will only give the details for some less obvious assertions. 

\subsection{Closed forms}\label{gjl123}

We consider the real {de} Rham complex as an object $\Omega\in \Fun^{desc}(\Mf^{op},\Ch)$.
For all $m\in\Z $ we define  $$\Diff^{m}(\R[0]):=\iota(\sigma^{\ge m}\Omega)\in \Fun^{desc}(\Mf^{op},\Ch[W^{-1}])\ ,$$ 
where $\sigma^{\ge m}\Omega$ is the truncation that discards all forms of degree $< m$, see Definition \ref{heute10}.

\begin{rem}\label{uuzzii10}{\rm  If $A\in \Fun^{desc}(\Mf^{op},\Ch)$ is a sheaf of chain complexes, then for $m\in \Z$ we let $Z^{m}(A)\in \Fun^{desc}(\Mf^{op},\Ab)$ be the abelian group valued sheaf of $m$-cycles. Furthermore, if $G\in  \Fun^{desc}(\Mf^{op},\Ab)$ is a sheaf of abelian groups, then  $G[k]\in  \Fun^{desc}(\Mf^{op},\Ch)$   denotes the sheaf of chain complexes with
$G$  in degree $-k$ for $k\in \Z$.
With this notation for $m\in \nat$  we have an equivalence $$ \Diff^{m}(\R[0])\simeq 
L(\iota(Z^{m}(\Omega)[-{m}]))$$ since the inclusion $$  Z^{m}(\Omega)[-{m}] \to 
 \sigma^{\ge m}(\Omega) $$ is a quasi-isomorphism  on stalks by the de Rham lemma. 
 In other words, the differential extension $\Diff^{m}(\R[0])$ of $\iota(\R[0])$ is represented by the sheaf of closed $m$-forms considered as  {a} one-component complex in degree $m$.
}\end{rem}

\begin{lem} \label{heu100}If we assume that $m\ge 1$, then we have
\begin{enumerate}
\item $\cZ(\Diff^{m}(\R[0]))\simeq \iota(\sigma^{\ge m}\Omega)$
\item $\cA(\Diff^{m}(\R[0]))\simeq \iota(\sigma^{\le m-1}(\Omega)[-1])$ 
\item $S(\Diff^{m}(\R[0]))\simeq 0$
\item $Z(\Diff^{m}(\R[0]))\simeq \iota(\R[0])$ 
\item $U(\Diff^{m}(\R[0]))\simeq \iota(\R[0])$
\item $\phi \simeq \id_{ \iota(\R[0])} $. 
\item {The integration map in degree $m$-cohomology
can be identified with the integration  of differential forms along $\Delta^{1}$
$$\int_{\Delta^{1}} : Z^{m}(\Omega(\Delta^{1}\times M))\to \Omega^{m-1}(M)/\im(d) .$$ }
\end{enumerate}
\end{lem}
\proof
We only show $7.$ by 
giving 
a very detailed specialization of the general construction of the integration map in order to show how the general construction leads to a concrete map involving the integration of forms.
We need to calculate the maps $\alpha_{\Diff^{m}(\R[0])}$ and
$\alpha_{\cH(\Diff^{m}(\R[0]))}$ and the filler in the square of \eqref{heut2000}.

We first fix our notation for the cone of a map of complexes $f:A\to B$. We set
$\Cone(f):=B\oplus A[1]$ with differential $d(b,a)=(db-f(a),-da)$.
Given a diagram
$$\xymatrix{A\ar[r]^{f}\ar[dr]^{0}&B\ar[d]^{\phi}\ar[r]&\Cone(f)\ar@{.>}[dl]^{\alpha}\\&C&}$$
the map $\alpha$ is given in terms of the filler $H$ with $dH+Hd=\phi\circ f$ of the left triangle  by
\begin{equation}\label{alphafiller}\alpha(b,a)=\phi(b)-H(a)\ .\end{equation} Similarly,
given a diagram
$$\xymatrix{&C\ar[dr]^{0}\ar[d]^{\phi}\ar@{.>}[dl]_{{\beta}}&\\
\Cone(f)[-1]\ar[r]&A\ar[r]^{f}&B}$$
the map $\beta$ is given in terms of the filler $L$ with $dL+Ld=f\circ \phi$ of the right triangle by
\begin{equation}\label{betafiller}\beta(c)=(L(c),\phi(c))\ .\end{equation}

Finally, if $f$ is an equivalence, and $Q$ is a contraction of the cone, written in components as $Q(b,a)=(Q_{0}(b,a), Q_{1}(b,a))$, {then the map} $g:B\to A$ {defined by} \begin{equation}\label{q1def}g(b):=-Q_{1}(b,0)\end{equation} is an inverse of $f$ and \begin{equation}\label{p1def}P:b\mapsto Q_{0}(b,0)\end{equation} is the homotopy such that
$dP+P d=\id-f\circ g$.

We now make the relevant piece of the diagram \eqref{heut300010} for $\Diff^{m}(\R[0])$ explicit:\begin{equation}\label{heut3000101}\xymatrix{&\iota(\Cone(\sigma^{\ge m}\Omega\to \Omega)[-1])\ar[d]&\\  0\ar[r]^{s_{\Diff^{m}(\R[0])}}&\iota(\sigma^{\ge m}\Omega) \ar[d]^{I}\ar[r]^-{R}&\iota(\Cone(0\to \sigma^{\ge m}\Omega))\\
& \iota(\Omega)&}\end{equation}
where we write the target of $R$ as a cone on purpose. 
The map $\alpha_{\iota(\sigma^{\ge m}\Omega)}$ is now given by the diagram
\begin{equation}\label{jaja30}\xymatrix{ 0\ar[r]\ar[d]\ar@/^{-1cm}/[d] &I\iota(\sigma^{\ge m}\Omega)\ar[r]\ar[d]^{\partial_{1}^{*}-\partial_{0}^{*}}&I\iota(\Cone(0\to \sigma^{\ge m}\Omega))\ar[dl]^{\alpha_{\iota(\sigma^{\ge m}\Omega)}}\\0\ar[r] 
&\iota(\sigma^{\ge m}\Omega)&}\ .\end{equation}
In explicit terms using \eqref{alphafiller} we get $\alpha_{\iota(\sigma^{\ge m}\Omega)}(b,a)=(\partial_{1}^{*}-\partial_{0}^{*})b$.

We now  consider the corresponding part of diagram  \eqref{heut300010}  for $\iota(\Omega)$:\begin{equation}\label{heut3000102}\xymatrix{  
\iota(\Omega)\ar[r]^{s_{\iota(\Omega)}}&\iota(\Omega)\ar[r]^-R\ar[d]^{I}&\iota(\Cone(\Omega\to \Omega))\\ 
& \iota(\Omega)&}\ .\end{equation} 
The map $\alpha_{\iota(\Omega)}$ is given by  
\begin{equation}\label{jaja31}\xymatrix{I\iota(\Omega)\ar[r]\ar[d]^{\partial_{1}^{*}-\partial_{0}^{*}}\ar@/^{-1cm}/[d]^{0} &I\iota(\Omega)\ar[r]\ar[d]^{\partial_{1}^{*}-\partial_{0}^{*}}&\iota(\Cone(I\Omega\to I\Omega))\ar[dl]^{\alpha_{\iota(\Omega)}}\\ \iota(\Omega)
\ar[r]&\iota(\Omega)&}\ .\end{equation}
In order to obtain the filler $H_{\iota(\Omega)}$ of the left two-gon we must find the preferred homotopy in \textcircled1 of \eqref{kp21}.
For this
we consider the map $\pr^{*}:\Omega\to I\Omega$. This map is an equivalence
and we can choose a contraction of its cone.

\begin{rem}{\rm There is a contractible space of 
choices for such contractions. It is at this point that we fix a particular choice related to the integration of forms.
If we would make a different choice here, then we would get another formula for the integration map, possibly not related to integration of forms. The action of the integration map in cohomology does not depend on the choices.
} \end{rem} 

We will take the contraction $Q_{\iota(\Omega)}$ of the cone of $\pr^{*}$ given by $Q_{\iota(\Omega)}(b,a):=(\bH(b),-\partial_{0}^{*}b)$,
where
$$\bH:I\Omega\to I\Omega \ , \quad \bH(\omega)(t):=\pr^{*}\int_{[0,t]}\omega_{|[0,t]\times M}\ , \quad  \omega\in \Omega(\Delta^{1}\times M)$$
is such that $d\bH+\bH d=\id-\pr^{*}\circ \partial_{0}^{*}$. By \eqref{q1def} and \eqref{p1def}
  it provides the required {preferred} homotopy $P_{\iota(\Omega)}:b\mapsto \bH(b)$ such that $dP_{\iota(\Omega)}+P_{\iota(\Omega)}d=\id-\pr^{*}\circ \partial_{0}^{*}$ and exhibits $\partial_{0}^{*}$ as the inverse of $\pr^{*}$.

 By composition it induces the homotopy
 $H_{\iota(\Omega)}:= (\partial_{1}^{*}-\partial_{0}^{*})\circ \bH$ such that $dH_{\iota(\Omega)}+H_{\iota(\Omega)}d=\partial_{1}^{*}-\partial_{0}^{*}$ which fills the left two-gon in {\eqref{jaja31}} as required. From \eqref{alphafiller} we get $\alpha_{\iota(\Omega)}(b,a)=(\partial_{1}^{*}-\partial_{0}^{*})(b)-H_{\iota(\Omega)}(a)$.

Note that if we apply the same construction to $\iota(\sigma^{\ge m}\Omega)$, then $P_{\iota(\sigma^{\ge m}\Omega)}$, $Q_{\iota(\sigma^{\ge m}\Omega)}$ and $H_{\iota(\sigma^{\ge m}\Omega)}$ naturally are the zero maps on the zero spaces.

We now observe that the map of morphisms
$(0\stackrel{s_{\iota(\sigma^{\ge m})\Omega}}{\longrightarrow} \sigma^{\ge m}\Omega) \to (\Omega\stackrel{s_{\iota(\Omega)}}{\to} \Omega)$ naturally induces maps of all diagrams used in the preceding constructions which also preserve the respective homotopies.
As a result we see that the natural filler of the  square {\textcircled{1}} in 
$$\xymatrix{&I\iota(\Cone(0\to \sigma^{\ge m}\Omega))\ar@{..>}[ld]_{\tilde \int}\ar[d]^{\alpha_{\iota(\sigma^{\ge m}\Omega)}}\ar[r]&I\iota(\Cone(\Omega\to \Omega))\ar[d]_{\alpha_{\iota(\Omega)}}\ar@/^{1cm}/[d]^{0}\\\iota(\Cone(\sigma^{\ge m}\Omega\to \Omega)[-1])\ar[r]&
\iota(\sigma^{\ge m}\Omega)\ar[r]&\iota(\Omega)
\save[]+<-2.1cm,0.8cm>*\txt<8pc>{\textcircled{1}}\restore
\save[]+<0.5cm,0.8cm>*\txt<8pc>{\textcircled{2}}\restore
}$$
is realized by the zero map.

In order to construct the map $\tilde \int$  we have to use the filler of the right two-gon {\textcircled{2}}
induced by the homotopy invariance of $\iota(\Omega)$. {Equivalently} we can use the contraction of the cone at the right upper corner given by 
$C(b,a):=(0,-b)$ which gives by composition with $\alpha_{\iota(\Omega)}$ the homotopy
$\tilde H:(b,a)\mapsto H_{\iota(\Omega)}(b)$ such that $d\tilde H+\tilde H d=\alpha_{\iota(\Omega)}$. The filler $L$ of the whole lower square (\textcircled1 and \textcircled2) is now given by 
$L(b,0):=H_{\iota(\Omega)}(b)$. From \eqref{betafiller} we get for the integration
$$\tilde \int(b,0)=(H_{\iota(\Omega)}(b),(\partial_{1}^{*}-\partial_{0}^{*})b)\ . $$
We finally trace this map through the diagram
$$\xymatrix{
\iota(\sigma^{\ge m}\Omega)\ar[d]^{\int}\ar[rr]^-{b\mapsto (b,0)}_-{\simeq} &&
\iota(\Cone(0\to \sigma^{\ge m}\Omega))\ar[d]^{\tilde \int} \\
\iota(\sigma^{\le m-1}\Omega) &&
\iota(\Cone(\sigma^{\ge m}\Omega\to \Omega)[-1])\ar[ll]_-{[b]_{\le m-1}\mapsfrom (b,a)}^-\simeq
}\ .$$
So the final formula for the integration comes out as
$$\int:\iota(\sigma^{\ge m} \Omega)\to \iota(\sigma^{\le m-1}\Omega)[-1]\ , \quad  b\mapsto [(\partial_{1}^{*}-\partial_{0}^{*})\circ \bH(b)]_{\le m-1}\ .$$
If we insert the definition of $\bH$, then we get the assertion. {Hence we completed the proof of Lemma \ref{heu100}.} \hB 

Observe that $\Diff^{m}(\R[0])$ is a pure differential refinement  of $\iota(\R[0])$ if and only if $m\ge 1$.

\subsection{Closed forms with values in a complex}  

For a chain complex $C\in \Ch_{{\R}}$ of real vector spaces    and $m\in \nat$ we define
\begin{equation}\label{ggsdhauzduiedeuwdef}\Diff^{m}(C):={\iota(\sigma^{\ge m}(\Omega\otimes_{\R}C))}\ .\end{equation}

\begin{lem}\label{heu101}We have
\begin{enumerate}
\item $\cZ(\Diff^{m}(C))\simeq {\iota\big(\sigma^{\ge m}(\Omega\otimes_{\R}\sigma^{\leq m-1}C)\big)}$ 
\item $\cA(\Diff^{m}(C))\simeq \iota\big(\sigma^{\le m-1}(\Omega\otimes_{\R} C){[-1]}\big)$
\item $S(\Diff^{m}(C))\simeq {\iota(\sigma^{\ge m}(C))}$
\item $Z(\Diff^{m}(C))\simeq {\iota(\sigma^{\le m-1}(C))}$
\item $U(\Diff^{m}(C))\simeq \iota(C)$
\item ${\phi}\simeq{\iota(C\to \sigma^{\le m-1}(C))} $. 
\item The integration map in degree $m$ is given by 

\begin{equation}\label{heu200}\int\beta=\left[\int_{\Delta^{1}} \beta\right]\ ,\end{equation}
where
$$\beta \in(\Omega\otimes_{\R}\sigma^{\leq m-1}C)^{m}(\Delta^{1}\times M)$$ represents an element of $H^{m}(\cZ(\Diff^{m}(C))(\Delta^{1}\times M)$, and the result is interpreted in
$(\Omega\otimes_{\R}C)^{m-1} (M)/\im(d)$.
\end{enumerate}
\end{lem}
\proof
The proof of $7.$ is similar to the corresponding assertion of Lemma \ref{heu100}. \hB 

This Lemma in particular shows that  $\Diff^{m}(C)$ is a differential refinement of $\iota(C)$.
For $k\in \Z$ we  have an equivalence  $$\Diff^{m}(C[k]) \simeq \Diff^{m+k}(C)[k]\ .$$

 \subsection{Generalized Deligne cohomology}\label{hjw1} 
 
  For a chain  complex  $C\in \Ch$   we define 
 $$\Diff^{m}(C\to C\otimes \R)\in \Fun^{{desc}}(\Mf^{op},\Ch[W^{-1}])$$ by the pull-back   \begin{equation}\label{heute2}\xymatrix{\Diff^{m}(C\to C\otimes \R)\ar[r]\ar[d]&\Diff^{m}(C\otimes \R)\ar[d]\\
\const(\iota(C))\ar[r]&\const(\iota(C\otimes \R))}\ . \end{equation}
\begin{lem}\label{rrq1}We have
\begin{enumerate}
\item $\cZ(\Diff^{m}(C\to C\otimes \R))\simeq {\iota\big(\sigma^{\ge m}(\Omega\otimes \sigma^{\leq m-1}C)\big)}$ 
\item $\cA(\Diff^{m}(C\to C\otimes \R))\simeq \iota\big(\sigma^{\le m-1}(\Omega\otimes C){[-1]}\big)$
\item {$S(\Diff^{m}(C\to C\otimes \R))$ is given by the pull-back
$$\xymatrix{S(\Diff^{m}(C\to C\otimes \R))\ar[d]\ar[r]& \iota(\sigma^{\ge m}(C\otimes \R))\ar[d]\\\iota(C)\ar[r]&\iota(C\otimes \R)
}\ .$$ }
\item $Z(\Diff^{m}(C\to C\otimes \R))\simeq \iota({\sigma^{\le m-1 }(C\otimes \R)})$
\item $U(\Diff^{m}(C\to C\otimes \R))\simeq \iota(C)$
\item ${\phi}\simeq{\iota(C\to C\otimes \R\to \sigma^{\le m-1 }(C\otimes \R))}$ is the natural map. 
\item The formula for the integration map is the same as in Lemma \ref{heu101}.
\end{enumerate}
\end{lem}
We see that $\Diff^{m}(C\to C\otimes \R)$ is a differential refinement of $\iota(C)$.

 \subsection{The Hopkins-Singer example}\label{ttz12}

Let $\Sp$ denote the stable $(\infty,1)$-category of spectra. This category is related to the $(\infty,1)$-category of chain complexes by
 the Eilenberg-MacLane functor
$$H:\Ch[W^{-1}]\stackrel{\sim}{\to} \Mod_\Sp(H\Z)\to \Sp\ $$
which preserves limits and hence  {restricts to a functor on sheaves} by objectwise application.
To  a spectrum $E$, a chain complex $C\in \Ch_\R$ of real vector spaces, a map $c:E\to H(\iota(C))$ and an integer $m\in\Z$ we associate the sheaf of spectra
$$\Diff^{m}(E,C,c)\in \Fun^{{desc}}(\Mf^{op},\Sp)$$ which is defined by the pull-back
\begin{equation}\label{phj1}\xymatrix{\Diff^{m}(E,C,c)\ar[r]^{\hat \cR}\ar[d]&H(\Diff^{m}(C))\ar[d]\\
\const(E)\ar[r]^-{\const(c)}&\const(H(\iota(C)))}\ .\end{equation}
This is essentially the Hopkins-Singer construction \cite{MR2192936} and was introduced in this form in \cite{2013arXiv1306.0247B}.
\begin{rem}{\rm 
The diagrams \eqref{heute2} and \eqref{phj1} must not be confused with the diagram \eqref{heute1}. In \eqref{heute2} and \eqref{phj1} the right upper corner is not pure in general and therefore does not represent the cycle data for the respective  left upper corner. One should also not confuse the maps $\hat \cR$ and $R$ which
have different targets.}
\end{rem}  

\begin{lem}\label{loi13}We have
\begin{enumerate}
\item $\cZ(\Diff^{m}(E,C,c))\simeq {H\iota\big(\sigma^{\ge m}(\Omega\otimes_{\R}\sigma^{\leq m-1}C) \big)} $.
\item $\cA(\Diff^{m}(E,C,c))\simeq H\big(\iota(\sigma^{\le m-1}(\Omega\otimes_{\R} C){[-1]})\big)$
 \item
 {$S(\Diff^{m}(E,C,c))$ is given by the pull-back
$$\xymatrix{S(\Diff^{m}(E,C,c))\ar[d]\ar[r]& H\iota(\sigma^{\ge m}(C))\ar[d]\\ E\ar[r]&H \iota(C)
}\ .$$ }
\item $Z(\Diff^{m}(E,C,c))\simeq H(\iota{(\sigma^{\le m-1 }(C))})$
\item $U(\Diff^{m}(E,C,c))\simeq {E}$ 
\item ${\phi\simeq(E\stackrel{c}{\to} H(\iota(C)) \to H(\iota(\sigma^{\le m-1}C)))}$. 
\item {The formula for the integration map is the same as in Lemma \ref{heu101}.}
  \end{enumerate}
\end{lem}
{\rm
Therefore $\Diff^{m}(E,C,c)$ is a differential refinement of $E$. Differential refinements of this form have been throughly investigated in \cite{skript}.  {From the definition of 
$\Diff^{m}(E,C,c)$ as a pull-back we get a commuting diagram
{\scriptsize \begin{equation}\label{ghb100}\hspace{-0.5cm}\xymatrix{&(\Omega\otimes_{\R}C)^{m-1}(M)/\im(d)\ar[rr]^{d}\ar[dr]^{a}&&Z^{m}((\Omega\otimes_{\R}C)(M))\ar[dr]&\\
 H^{m-1}(M;C)\ar[ur]\ar[dr]&&\Diff^{m}(E,C,c)^{m}(M)\ar[ur]^{\hat \cR}\ar@{>>}[dr]^{I}&& H^{m}(M;C)\\
&H^{m}(M;S)\ar@{^{(}->}[ur]\ar[rr]&&H^{m}(M;E)\ar[ur]^{c}}\ ,\end{equation}}where $d$ is the de Rham differential.     
This diagram differs from the differential cohomology diagram  \eqref{ghb1}  at the 
right upper corner.  {There is a {canonical} projection map}
$$R_{\Diff^{m}(C)}:Z^{m}((\Omega\otimes_{\R}C)(M))\to Z^{m}({(\Omega \otimes_\R \sigma^{\leq m-1} C})(M))$$ which in turn induces a map from \eqref{ghb100} to the  {abstract} differential cohomology diagram constructed in \eqref{ghb1}
{\scriptsize \begin{equation}\label{ghb100nnn}\hspace{-0.5cm}\xymatrix{&(\Omega\otimes_{\R}C)^{m-1}(M)/\im(d)\ar[rr]\ar[dr]^{a}&&
Z^{m}({\Omega \otimes_\R \sigma^{\leq m-1} C})(M)\ar[dr] &\\
 H^{m-1}(M;C)\ar[ur]\ar[dr]&&\Diff^{m}(E,C,c)^{m}(M)\ar[ur]^{R}\ar@{>>}[dr]^{I}&& H^{m}(M;C)\\
&H^{m}(M;S)\ar@{^{(}->}[ur]\ar[rr]&&H^{m}(M;E)\ar[ur]^{c}}\ ,\end{equation}}
Finally we want to deduce the usual homotopy formula for $\Diff^{m}(E,C,c)$ (see \cite[(1)]{MR2608479}) from the general one given by Lemma \ref{loi13}, $7.$ 
\begin{lem} \label{sam300} If $x\in \Diff^{m}(E,C,c)^{m}(\Delta^{1}\times M)$, then we have
\begin{equation}\label{heu300}\partial_{1}^{*}x-\partial_{0}^{*}x=a(\int_{\Delta^{1}}\hat \cR(x))\ .\end{equation}
\end{lem}
\proof
We have
$ \hat \cR(x)\in Z^{m}((\Omega\otimes_{\R}C)(M))$. Then
$R(x)=R_{\Diff^{m}(C)}(\hat \cR(x))$.
The formula \eqref{heu200} gives $\int\circ R(x)=[\int_{\Delta^{1}} \hat \cR(x)]$. The general homotopy formula 
\eqref{ghb3}  now gives
$\partial_{1}^{*}x-\partial_{0}^{*}x=a(\int_{\Delta^{1}} \hat \cR(x))$. This is the assertion. \hB
\begin{rem}{\rm Note that the usual argument for  {the} homotopy formula \eqref{heu300} (see e.g. \cite{MR2608479})
just uses the diagram \eqref{ghb100}. Further note the difference in the right-hand sides of 
\eqref{heu300} and {\eqref{ghb3}}.}\end{rem}
}}

\subsection{The geometric suspension construction} \label{suspension}
In this example we present a construction which produces new differential refinements from given ones.
Let $\bC$ again be an arbitrary stable, presentable $\infty$-category.
For a sheaf $\hat F\in \Fun^{desc}(\Mf^{op},\bC)$ we define  {(compare with \eqref{heut1000})}
$${I_{S^{1}}}\hat  F\in \Fun^{desc}(\Mf^{op},\bC) \qquad \mathrm{by} \qquad {(I_{S^{1}}\hat F)}(M):=\hat F(S^{1}\times M).$$
Let $1\in S^{1}$
be the base point and  $i_{1}:M \to S^{1}\times M$ the corresponding embedding. Then we define
$I_{\widetilde S^{1}}\hat F\in \Fun^{desc}(\Mf^{op},\bC)$
   by the fibre sequence 
$$I_{\widetilde S^{1}}\hat F \to {I_{S^{1}}}\hat F \to \hat F$$
If $\hat F$ is homotopy invariant, then
{by Proposition \ref{propo} it is of the form $\const(F)$ for   $F:=\ev_{*}(\hat F)$.   From this we easily get
   a natural equivalence
  $\Sigma I_{\widetilde S^{1}}\hat F\simeq  \hat F$.   
 Furthermore, 
 for a general $\hat F$,
 we get an equivalence
 \begin{equation} \label{t19okt}
  \cH(I_{\widetilde S^{1}}\hat F) \simeq I_{\widetilde S^{1}} \cH(\hat F)
 \end{equation}
 as a consequence of Lemma \ref{thop}.
 
 \begin{kor}\label{tui1}
 If $\hat F$ is a differential refinement of $F\in \bC$, then
 $\Sigma I_{\widetilde S^{1}}\hat F$ is also a differential refinement of $F$.
 \end{kor}

\begin{lem}\label{rrq1333} 
\begin{enumerate}
 \item $\cZ(I_{\widetilde S^{1}}\hat F)\simeq 
 {\Fib(\cZ(I_{S^1}\hat F)\to \cZ(\hat F))}$.

\item $\cA(I_{\widetilde S^{1}}\hat F)\simeq I_{\widetilde S^{1}} \cA(\hat F) $
\item $S(I_{\widetilde S^{1}}\hat F))\simeq {\Fib(\hat F(S^1)\to \hat F(*)) }$
\item $Z(I_{\widetilde S^{1}}\hat F)\simeq  \Cone({S(I_{\widetilde S^{1}} \hat F)} \to { \Sigma^{-1}} U(\hat F)) $ 
 \item $U(I_{\widetilde S^{1}}\hat F)\simeq \Sigma^{-1}U(\hat F)$
\item $ \phi:\Sigma^{-1}U(\hat F)\to \Cone({S(I_{\widetilde S^{1}} \hat F)}\to {\Sigma^{-1} U(\hat F)})$ is the natural map. 
\end{enumerate}
\end{lem}

\bigskip

By Definition \eqref{ggsdhauzduiedeuwdef} we have $\Diff^{m}(C) = \iota(\sigma^{\ge m}(\Omega\otimes_{\R} C))$.
Integration over $S^{1}$ induces a natural  map of complexes
$$\int_{S^{1}\times M/M}:(\sigma^{\ge m}(\Omega\otimes_{\R} C)(S^{1}\times M))[1]\to \sigma^{\ge m-1}(\Omega\otimes_{\R} C)(M)\ .$$   It induces the integration map 
 \begin{equation}\label{pjz40}\int:I_{\widetilde S^{1}}\Diff^{m}(C)[1] \to  \Diff^{m-1}(C)\ .\end{equation}

More generally, let us consider a spectrum $E\in \Sp$, a complex $C\in \Ch$ of real vector spaces, and a map $c:E\to H(\iota(C))$. 
Then we obtain a new differential refinement $\Sigma I_{\widetilde S^{1}}\Diff^{m}(E,C,c)$ of $E$.
It is again related with the Hopkins-Singer example $\Diff^{m-1}(E,C,c)$ (note the shift in the superscript) by an  integration map   
$$\int:\Sigma I_{\widetilde S^{1}}\Diff^{m}(E,C,c)\to \Diff^{m-1}(E,C,c)$$ 
(induced by \eqref{pjz40}). This integration played an important role in the axiomatic characterization of differential cohomology} \cite{MR2608479}.

\bigskip 

We have seen that for homotopy invariant sheaves $F$ we have an equivalence  $\Sigma I_{\widetilde S^{1}}F\simeq F$. One could conversely ask
which consequences an equivalence of this type has for a general sheaf.  
We first claim that there is always a canonical morphism $ F \to \Sigma I_{\widetilde S^{1}}F$  
 which is natural in $F$. {Such a morphism} can equivalently be described as a map $  {\Sigma^{-1}} F \to I_{\widetilde S^{1}}F$ by stability of $\bC$. 
{We take the upper morphism in the map of fibre sequences.
$$\xymatrix{{\Sigma^{-1}} F\ar[r]\ar[d]&I_{\widetilde S^{1}}F\ar[d]\\
F^{S^{1}_{top}}\ar[r]^{\eqref{natnatnat}}\ar[d]&{I_{S^1}}F\ar[d]\\
F\ar@{=}[r]&F }\ ,$$
where the lower vertical maps are induced by the inclusion  of the point $1\in S^{1}$. 

\begin{prop}\label{sushom}
 For a sheaf $F\in \Fun^{desc}(\Mf^{op},\bC)$  the canonical morphism $ F \to \Sigma I_{\widetilde S^{1}}F$ is an equivalence if and only if $F$ is 
 homotopy invariant
\end{prop}
\begin{proof}
{We must show that the fact that the canonical map   $F \to \Sigma I_{\widetilde S^{1}}F$ is an equivalence} implies that $F$ is homotopy invariant.
{We} first assume that $F$ is  pure. In this case we conclude that 
$$ F(S^1) \simeq ({I_{S^1}} F)(*)  \simeq  (I_{\widetilde S^{1}}F)(*) \simeq (\Sigma^{-1} F)(*) \simeq 0 \ ,$$
 {where we have used that $F$ is pure for the second equivalence. By induction we get $$F(\underbrace{S^1 \times ... \times S^1}_{n\times})\simeq 0$$ for all $n\in \nat$.}  Now we apply Lemma \ref{test_equi} from the appendix to conclude that $F \simeq 0$. 
In particular $F$ is homotopy invariant. {We thus  have} proven the Proposition for pure sheaves.

Now let $F$ be an arbitrary sheaf {such that the canonical map   $ F \to \Sigma I_{\widetilde S^{1}}F$ is an equivalence}. Then it suffices to show that $\cZ(F)$ is homotopy invariant. {Indeed, the presentation of $F$    {by} the pullback diagram \eqref{tsep22} in Proposition \ref{wie4} implies that  $F$ is then       homotopy invariant, too.} In order to show that $\cZ(F)$ is homotopy invariant it suffices to show that 
the canonical morphism induces an equivalence $ \cZ(F) \simeq \Sigma I_{\widetilde S^{1}}\cZ(F)$ since $\cZ(F)$ is pure.
We use that the functor
$$ \Sigma I_{\widetilde S^{1}} :  \Fun^{desc}(\Mf^{op},\bC) \to \Fun^{desc}(\Mf^{op},\bC) $$
is exact. Using the sequence $\cS(F) \to F \to \cZ(F)$ {and the naturality of the canonical morphism} we are thus reduced to show that 
$  \cS(F) \to \Sigma I_{\widetilde S^{1}} \cS(F)$ is an equivalence. But since $\cS(F)$ is homotopy invariant, this is true.

\end{proof}
\begin{rem}{\rm
An integration map like \eqref{pjz40} is a differential replacement of a desuspension map in stable homotopy theory. The stable homotopy category embeds into the category of   sheaves of spectra on $\Mf$ as the full subcategory of homotopy invariant sheaves. Proposition \ref{sushom} asserts that this subcategory is characterized by the property that the suspension is an equivalence. Conversely the Proposition also shows that for real differential refinements one {cannot} hope that the integration map is an equivalence}
\end{rem}

\begin{rem}{\rm One can equivalently define the 
  category of homotopy invariant sheaves   as the localization of the category of all sheaves
at the morphisms of the form ${y(\Delta^{1})} \otimes F \to * \otimes F = F$, where $y(M)$ is the sheaf of spaces represented by $M$ (see \eqref{sdf1}), and   $F$  is a sheaf with values in $\bC$. Proposition \ref{sushom} now basically shows that we obtain the same category if we localize at the morphisms of the form ${y(S^1)} \otimes F\to S^1_{top} \otimes F $ . The same proof also shows that this is true for any {positive-dimensional} manifold $M$: the localization of the category of sheaves at the morphisms ${y(M)} \otimes F \to M_{top}\otimes F$ is equivalent to the category of homotopy invariant sheaves.}
\end{rem}

\subsection{Forms on the loop space}\label{uuzzii30} 

Here is a more exotic example derived from equivariantly closed forms on loop spaces.
 We consider the presheaf of chain complexes $F\in \Fun(\Mf^{op},\Ch)$ given by
$$F:M\mapsto (\Omega_{\C}(LM)[b,b^{-1}]^{S^{1}}\ , d+ b^{-1}i_{\xi})\ ,$$
 where $LM$ is the smooth loop space of $M$, $i_{\xi}$ inserts the fundamental vector field $\xi$ of the $S^{1}$-action on $LM$ by rotation of loops, and $\deg(b)=-2$.  
 The cohomology of $F(M)$ is a version of equivariant cohomology of $LM$. 
 We will also consider the variant given by
$$
F^{{\prime}}(M):=\Omega_{\C}(LM)[b^{-1}][[b]]^{S^{1}}
$$

\begin{lem}
The {presheaves} $\iota(F^{\sharp}) \in \Fun(\Mf^{op},\Ch[W^{-1}])$, $\sharp\in \{-, {\prime}\}$, are homotopy invariant.
\end{lem}
\proof The argument is the same in both cases. We consider the   case $\sharp=-$. 
We must show that
$\pr:\Delta^{1}\times M\to M$  induces an equivalence
$F(M)\to F(\Delta^{1}\times M)$ {for all manifolds $M$.}
We consider the inclusion $\partial_{0}:M\to \Delta^{1}\times M$.
Then
$\pr\circ \partial_{0}=\id$ and therefore the composition
$$F(M)\stackrel{\pr^{*}}{\to} F(\Delta^{1}\times M) \stackrel{\partial_{0}^{*}}{\to} F(M)$$
is the identity. It remains to show that
$$F(\Delta^{1}\times M)\stackrel{\partial_{0}^{*}}{\to} F(M)\stackrel{\pr^{*}}{\to} F(\Delta^{1}\times M)$$
is an equivalence, too. To this end we consider the map
$$\phi:\Delta^{1}\times L(\Delta^{1}\times M)\to L(\Delta^{1}\times M)\ , \quad
(t,\sigma,\gamma)\mapsto (t\sigma,\gamma)\ , $$ where we write a loop in $\Delta^{1}\times M$ as a pair of loops $(\sigma,\gamma)$.
For $i=0,1$ the compositions
$$L(\Delta^{1}\times M)\xrightarrow{\partial_{i}\times \id_{L(\Delta^{1}\times M)}} \Delta^{1}\times L(\Delta^{1}\times M)\stackrel{\phi}{\to} L(\Delta^{1}\times M)\ .$$
  are the composition
$\partial_{0}\circ \pr$ and $\id$, respectively.
We define a map
\begin{equation}\label{heut401}
\bH:\Omega_{\C}(L(\Delta^{1}\times M))\to \Omega_{\C}(L(\Delta^{1}\times M))[-1] \text{ by }\bH(\omega):=\int_{\Delta^{1}} \phi^{*}\omega \ .
\end{equation}
Then we have
$$[d,\bH](\omega)=\omega-\pr^{*}\partial_{0}^{*}\omega\ .$$
Now we observe that $\bH$ preserves the subspace of $S^{1}$-invariants   and (graded) commutes
with $i_{\xi}$. Therefore, $\bH$  {yields} a homotopy
from $\pr^{*}\circ \partial_{0}^{*}$ to $\id$ on $F$.
\hB 

\begin{lem}\label{uli1232}
For  $\sharp\in \{-, {\prime} \}$ the restriction to constant loops induces an equivalence
$$L(\iota(F^{\sharp}))\stackrel{\sim}{\to}\iota (\Omega_{\C}[b,b^{-1}]) \ .$$
 \end{lem}
\proof
Using Lemma \ref{lll24}   it suffices to check an equivalence of stalks 
$$\iota(F^{\sharp})(\R^{n}_{0})\stackrel{\sim}{\to}  \iota(\Omega_{\C}[b,b^{-1}])(\R^{n}_{0})\ ,\quad \forall n\in \nat\ .$$
Since both presheaves are homotopy invariant it actually suffices to consider the case $n=0$.
In this case we obviously get an equivalence. \hB 
}

Since  $\iota(\Omega_{\C}[b,b^{-1}])\simeq \const(\iota(\C[b,b^{-1}]))$ we see that
$L(\iota(F^{\sharp}))$ is {the}  constant sheaf on $ \iota(\C[b,b^{-1}])$. 
If $M$ is simply connected, then by \cite[Cor. V.3.3]{MR793184} or \cite[p. 342]{MR1113683} we have
$\iota(\C[b,b^{-1}])\simeq \iota(F(M))$, i.e. $\iota(F)$ (without sheafification)
is almost a constant presheaf.

\bigskip

{ For $\sharp\in \{-,{\prime}\}$
we define the sheaf
\begin{equation}\label{heute100}\hat F_{loop}^{\sharp}:=L(\iota( \sigma^{\ge 0}(F^{\sharp})))\in\Fun^{desc}(\Mf^{op},\Ch[W^{-1}])\ .\end{equation}
{\begin{rem}\label{remloop}{\rm
Note that sheafification here is a rather drastic operation, since there is no reason why differential forms on the loop space of a manifold $M$ should have some locality behaviour with respect to open sets in $M$. But we have to sheafify here in order to be able to proceed in the framework of the present paper. It is clear that in a more suitable language one should take another locality behaviour into account which is related to field theories and fusion products on loop spaces, see e.g.\cite{waltra}. }
\end{rem}}

 \begin{lem}\label{uuzzii2} For $\sharp\in \{-,{\prime}\}$
 the sheaf $\hat F_{loop}^{\sharp}$ is a differential refinement of $\iota(\C[b,b^{-1}])$.
\end{lem}
\proof
The argument is the same in both cases so that we discuss the case $\sharp=-$.
We consider $F$ as a presheaf of $\cC^{\infty}$-modules via the evaluation at $1\in S^{1}$.
By Lemma \ref{hj12} we have $\iota(L(F))\in \Fun^{desc}(\Mf^{op},\Ch[W^{-1}])$ since $L(F)$ is a sheaf of complexes  whose components are sheaves of $\cC^{\infty}$-modules. We now claim that 
the natural map  induces an  equivalence
\begin{equation}\label{uli1231}L(\iota(F))\stackrel{\sim}{\to}  \iota(L(F))\ .\end{equation}  
By Lemma \ref{lll24} this can be checked on stalks, i.e. we must show that
$\iota(F)(\R^{n}_{0})\to \iota(L(F))(\R^{n}_{0})$ is an equivalence for all $n\in \nat$.
Since $\iota:\Ch\to \Ch[W^{-1}]$ preserves filtered colimits (Lemma \ref{heut1}) it suffices to check that
$F(\R^{n}_{0})\to L(F)(\R^{n}_{0})$ is an equivalence in $\Ch$ which is clearly the case. By a similar argument we get the first equivalence in the following chain
\begin{equation}\label{uli2340}L(\iota( \sigma^{\ge 0}(F )))\simeq L(\iota( \sigma^{\ge 0}(L(F )))) \stackrel{Lemma \ref{hj12}}{\simeq}  \iota(\sigma^{\ge 0} (L(F)))\ .\end{equation}
By Lemma \ref{uuzzii1} we get $\cH(\iota(\sigma^{\ge 0} L(F)))\simeq \cH(\iota(L(F)))$. 
This implies  the second equivalence in 
\begin{eqnarray*}\hat F_{loop}= L(\iota( \sigma^{\ge 0}(F )))\stackrel{\eqref{uli2340}}{ {\simeq} } \cH(\iota(\sigma^{\ge 0} L(F))) {\simeq}  \cH(\iota(L(F)))\stackrel{\eqref{uli1231}}{ {\simeq} } \cH(L(\iota(F)))&&\\
\stackrel{Lemma \ref{uli1232}}{ {\simeq} }\cH(\iota(\Omega_{\C}[b,b^{-1}])) {\simeq}  \const(\iota(\C[b,b^{-1}]))&&\ . \end{eqnarray*} 
 \hB
}

The sheaf $L(F^{\sharp})$ has an increasing filtration by $b$-degree
such that
$\cF^{p}L(F^{\sharp})$ is the sheafification of the presheaf
$M\mapsto b^{p}{\Omega_{\C}(LM)[b^{-1}]^{S^{1}}}$. 
 \begin{lem}\label{uuzzii4}$\ $ 
\begin{enumerate}
\item $\cZ(\hat  F^{\sharp}_{loop}) {\simeq}  \iota(\Cone( \cF^{{0}}L(F^{\sharp }) \to {L(}\sigma^{\ge 0}(F^{\sharp} ))))$        
\item $\cA(\hat  F^{\sharp}_{loop}) {\simeq}  \iota(\sigma^{\le -1 }(L( F^{\sharp}))[-1])$
\item  $S(\hat F^{\sharp}_{loop}) {\simeq}  \iota(\C[b^{-1}])$ 
\item $Z(\hat F^{\sharp}_{loop}) {\simeq}  \iota(b\C[b])$ 
\item $U(\hat F^{\sharp}_{loop}) {\simeq}  \iota(\C[b,b^{-1}])$
\item $\phi \simeq \iota( \C[b,b^{-1}] \to b\C[b])$ is the natural projection.
\item 
The integration in degree $m$-cohomology is given by
   $\int\beta= [\partial_{1}^{*}\bH(\beta)]$,
   where $\beta\in \cZ(\hat  F^{\sharp}_{loop})^{0}(\Delta^{1}\times M)$
 , $\bH$ is as in \eqref{heut401}, and
the result is interpreted in $F^{\sharp,-1}(M)/\im(d)$. 
\end{enumerate}
\end{lem}
\proof The argument for $7.$ is similar to the corresponding one in the proof of \eqref{heu100}.   \hB

\subsection{A no forms example}\label{uuzzii60} 

In view of the examples presented so far one could get the impression that the $\cA$- or $\cZ$-piece of a sheaf is always given by some construction with differential forms or at least has some sort of $\R$-module structure. This is, of course, not the case.
Take
$\hat E:=\cG(H\Z/2\Z)$, where $\cG$ is the Godement functor \eqref{mikoe}.
This is an example of a sheaf $\hat E\in \Fun^{desc}(\Mf^{op},\Sp)$
such that $\cA(\hat E)\not\simeq 0$, but $\cA(\hat E)\wedge H\Q\simeq 0$.

\section{Examples from  Lie groups} \label{ttz30}

\newcommand{\bbB}{\mathbb{B}}
\newcommand{\gaaa}{\mathfrak{g}}
\newcommand{\bz}{\mathbf{z}}
\newcommand{\Sing}{{\mathrm{Sing}}}

The Grothendieck topology of $\Mf$ is subcanonical which means that   the Yoneda embedding sends a smooth manifold $M\in \Mf$ to a sheaf of sets 
$$Y(M)\in \Fun^{desc}(\Mf^{op},\Set)\ .$$  We consider the functor 
\begin{equation}\label{vm10}i:\Set\to \sSet\stackrel{\iota}{\to} \sSet[W^{-1}]\end{equation} 
which maps a set to the corresponding constant simplicial set.
This functor preserves limits. 
Consequently we get a sheaf of simplicial sets
 \begin{equation}\label{sdf1}y(M) :=i(Y(M))\in \Fun^{desc}(\Mf^{op},\sSet[W^{-1}])\ ,\end{equation}
i.e. the Grothendieck topology on $\Mf$ is subcanonical in the  $(\infty,1)$-{categorical} sense.  

Further note for later use that 
  for a simplicial set
  $A\in \Fun(\Delta^{op},\Set) = \sSet $ we have \begin{equation}\label{zur12}\colim_{\Delta^{op}} i(A) {\simeq}  \iota(A)\ .\end{equation} 
 
We  let  $M^{\delta}\in \sSet[W^{-1}]$ denote the underlying set of $M$ considered as a discrete simplicial set.

\begin{lem}\label{wie2}
We have natural equivalences
$${U(y(M))\simeq {M_{top}}\ {\text{ and }} \quad {S(y(M))} {\simeq}  M^{\delta}\ .}$$
\end{lem}
\proof
We use the equivalence $\cH(y(M))\simeq L \circ \bs(y(M))$ shown in Lemma \ref{lll25}, where $\bs$ is defined in \eqref{poi1}. By \eqref{dsf12} we have an equivalence.
$${M_{top}:=}\sing(M){\simeq}\bs(y(M))(*)\in \sSet[W^{-1}]\ .$$ 
We also note that for a presheaf $F$ we have $F(*) \simeq L(F)(*) $. Together this gives the first equivalence.
    The second equivalence immediately follows from \eqref{vm1}.    
\hB

We let $\Groupoids$ denote the (one-)category of groupoids and $W\subset \Mor(\Groupoids)$ be the equivalences. {The nerve functor $\Groupoids\to \sSet$ maps equivalences to homotopy equivalences  and therefore induces a functor denoted by the same symbol $\Nerve:\Groupoids[W^{-1}]\to \sSet[W^{-1}]$.}  We  consider a Lie group $G$  and  form the {sheaves}
$$\Bun(G)\ , \Bun(G)^{\nabla}\in \Fun^{desc}(\Mf,\Groupoids[W^{-1}])$$
which associate to a manifold $M$  the groupoids of $G$-principal bundles 
and $G$-principal bundles with connection.
 We consider the objects
$$\bbB G:=\Nerve(\Bun(G))\ ,\quad \bbB G^{\nabla}:=\Nerve(\Bun(G)^{\nabla}) $$ in  $\Fun^{desc}(\Mf^{op},\sSet[W^{-1}])$. 
The nerve of the Lie groupoid $\big(G{\rightrightarrows} *\big)$ formed internally to $\Mf$ is a simplicial manifold
$BG^{\bullet}$. The simplicial set
$$BG:=\colim_{\Delta^{op}} \sing(BG^{\bullet})\in \sSet[W^{-1}]$$
realizes the homotopy type of the classifying space of $G$.

  \begin{lem}\label{loi1}
We have natural equivalences $${U(\bbB G )\simeq   U(\bbB G^{\nabla})\simeq BG} $$
and
 $${S(\bbB G)}\simeq {S(\bbB G^{\nabla})}\simeq {BG^{\delta}} \ . $$
 \end{lem}
 \proof  
 The second assertion is {a straightforward consequence of \eqref{vm1}} so that we concentrate on the first.
We first recall the presentation of the  {sheaf} $\Bun(G)$ using the atlas
 $*\to \Bun(G)$. It gives rise to a sheaf of groupoids
$$(Y(G){\rightrightarrows} *)\in \Fun^{desc}(\Mf^{op},\Groupoids)\ .$$  
The sheafification of {its image in the localization}
$$ L(\iota(Y(G){\rightrightarrows} *))\in \Fun^{{desc}}(\Mf,\Groupoids[W^{-1}])$$ 
is equivalent to  $\Bun(G)$.
Since we can write the sheafification $L$ as an  {iterated} filtered colimit {of limits},
 and the nerve functor $\Nerve:\Groupoids[W^{-1}]\to \sSet[W^{-1}]$   commutes with limits and filtered colimits,  it commutes with sheafification.
 We therefore have an equivalence
$$\bbB G\simeq L(\Nerve(\iota(Y(G){\rightrightarrows} *)))\ .$$
We now observe that 
$$\Nerve(\iota(Y(G){\rightrightarrows} *)) \simeq \iota(Y(BG^{\bullet}))\ .$$
 Therefore from \eqref{zur12}   we get
$$\Nerve(\iota(Y(G){\rightrightarrows} *))\simeq \colim_{\Delta^{op}} y(BG^{\bullet})\ .$$
Since $\cH$ commutes with colimits we see that
$$\cH(\Nerve(\iota(Y(G){\rightrightarrows} *)))\simeq  \colim_{\Delta^{op}} \cH(y(BG^{\bullet}))\simeq \const(\colim_{\Delta^{op}} \sing( BG^{\bullet}))\simeq \const( BG )\ .$$

We now consider the  {sheaf} $\bbB G^{\nabla}$.
The sheaf of groups $Y(G)$ acts on the sheaf of sets $\Omega^{1}\otimes \gaaa$ by 
$$g \cdot\omega=\Ad(g)\omega-dgg^{-1}\ .$$ We obtain the action groupoid
$$(Y(G)\times \Omega^{1}\otimes \gaaa{\rightrightarrows} \Omega^{1}\otimes \gaaa)\in\Fun^{desc}(\Mf^{op},\Groupoids)\ .$$ {The sheafification of its localization}
$$L(\iota(Y(G)\times \Omega^{1}\otimes \gaaa{\rightrightarrows} \Omega^{1}\otimes \gaaa))\in \Fun^{{desc}}(\Mf,\Groupoids[W^{-1}])$$   is
equivalent to $\Bun(G)^{\nabla}$. We now argue similarly as in the case of $\bbB G$ using in addition that
$\cH(i(\Omega^{1}\otimes \gaaa))\simeq *$ by a similar argument as for Lemma \ref{zzz2}. \hB

\subsection{Examples from abelian Lie groups}

A commutative Lie group $A$ gives rise to a sheaf of  grouplike commutative monoids $y(A)$. We start with explaining this notion and its relation to spectra.
In the following 
we use the cartesian symmetric monoidal structure in order to talk about commutative monoids in various $(\infty,1)$-categories.
In the case  of spaces $\sSet[W^{-1}]$ we call
  a commutative monoid   $M\in \Comm\Mon(\sSet[W^{-1}])$ grouplike, if the monoid $\pi_{0}(M)$ is a group. 
We let $$\Comm\Grp(\sSet[W^{-1}])\subseteq \Comm\Mon(\sSet[W^{-1}])$$ be the full subcategory of grouplike monoids. 
It is well-known that the  infinite loop space functor induces an equivalence of $(\infty,1)$-categories
$\Omega^{\infty}:\Sp^{\ge{0}}\stackrel{\sim}{\to} \Comm\Grp(\sSet[W^{-1}])$, where $\Sp^{\ge 0}\subset \Sp$ denotes the full subcategory of connective spectra.
We consider the composition
\begin{equation}\label{wie3}\sp:\Comm\Grp(\sSet[W^{-1}])\stackrel{(\Omega^{\infty})^{-1}}{\longrightarrow} \Sp^{\ge 0}~ {\subset}~\Sp\ .\end{equation} 
In particular,
if $A$ is an {ordinary} abelian group, then we have an equivalence
\begin{equation}\label{wie1}\sp(i(A))\simeq H(\iota(A[0])) .\end{equation}
Here the right hand side denotes the Eilenberg-MacLane spectrum on $A$.
Note that  $\sp$ does not preserve limits. Consequently, if we apply this functor to a sheaf of grouplike monoids, we must sheafify the result again. 

Let now $A$ be an abelian Lie group. Then $y(A)\in \Fun^{desc}(\Mf^{op},\Comm\Grp(\sSet[W^{-1}]))$.
{We} define
$$A_{\infty}:=L(\sp(y(A)))\in  \Fun^{desc}(\Mf^{op},\Sp)\ .$$ 
Note that ${A_{top}}\in \Comm\Grp(\sSet[W^{-1}])$ so that we can apply \eqref{wie3}.
\begin{lem}
 We have equivalences
 $$U(A_{\infty})\simeq \sp({A_{top}})\ {\text{ and }} \quad S(A_{\infty})\simeq H(\iota(A^{\delta}[0]))\ .$$
\end{lem}
\proof
We use that {$U$} commutes with $L\circ \sp$ and that $U(y(A))\simeq A_{top}$ in order to deduce the first equivalence. For the second we
use the second equivalence in Lemma \ref{wie2} and  \eqref{wie1}. \hB

The most interesting example is the group $U(1)$. In the following Lemma we observe that the sheaf $U(1)_{\infty}$ is not new but equivalent to a special case of Example \ref{hjw1}.
\begin{lem}
We have an equivalence
$$U(1)_{\infty}\simeq H(\Diff^{{0}}(\Z[1] \to \R[1]))\ .$$ 
\end{lem}
\proof
We show the Lemma by analysing the data required in Proposition \ref{wie4} in order to characterize $U(1)_{\infty}$ and compare {it} with the corresponding data for $H(\Diff^{{0}}(\Z[1]\to \R[1])))$.
 The short {exact} sequence {of Lie groups} 
$$0\to \Z\to \R\to U(1)\to 0$$
yields a fibre sequence
$$\Z_{\infty}\to \R_{\infty}\to U(1)_{\infty}\stackrel{\partial}{\to} \Sigma\Z_{\infty}\ .$$
{Note that $\Sigma^k\Z_\infty$ is homotopy invariant and equivalent to $\const(H(\iota(\Z[k])))$.}
We apply the homotopification $\cH$ {to the fibre sequence} and use that $\cH(\R_{\infty})\simeq 0$ by Lemma \ref{zzz2} in order to conclude that
$$\cA(U(1)_{\infty})\simeq \R_{\infty}\simeq H(\iota(\sigma^{\le 0}\Omega))\ ,$$
{and that   map $\partial$ induces an equivalence $$\cH(U(1)_{\infty})\simeq \const(H(\iota(\Z[1])))\ .$$ Hence $U(1)_{\infty}$ is a differential refinement of $H(\iota(\Z[1]))$.}
{Moreover,} $\cZ(U(1)_{\infty})$ is determined by the sequence
$${\cS(H(\iota(\sigma^{\le 0}\Omega)))}\to H(\iota(\sigma^{\le 0}\Omega))\to \cZ(U(1)_{\infty})\to {\cS(\Sigma 
H(\iota(\sigma^{\le 0}\Omega))}\ .$$
Using that $${\cS(\iota(\sigma^{\le 0}\Omega))}\simeq \const(H(\iota(\R[0])))\simeq H(\iota(\Omega))$$ we get
$$\cZ(U(1)_{\infty})\simeq H(\iota(\sigma^{\ge 1}\Omega)[1])\ .$$ This implies that
$Z(U(1)_{\infty})\simeq H(\iota(\R[1]))$ and
${\phi}:H(\iota(\Z[1]))\to H(\iota(\R[1]))$ is induced by  inclusion of $\Z$ into $\R$. The result now follows from comparison with 
Lemma \ref{rrq1}. \hB

If $A$ is an abelian Lie group, then  the tensor product induces a symmetric monoidal structure on the  {sheaves} $\Bun(A)$ and $\Bun(A)^{\nabla}$,
 i.e. we have
$$\Bun(A)\ , \Bun(A)^{\nabla}\in \Fun^{desc}(\Mf,\Comm\Mon(\Groupoids[W^{-1}]))\ .$$
Since every $A$-principal bundle (with connection) is tensor invertible we conclude that
$$\bbB A\ ,\bbB A^{\nabla}\in \Fun^{desc}(\Mf,\Comm\Grp(\sSet[W^{-1}]))\ .
$$
 
We define the objects
\begin{equation}\label{uuzzii21}BA_{\infty}:=L(\sp(\bbB A))\ , \quad BA_{\infty}^{\nabla}:=L(\sp(\bbB A^{\nabla}))\end{equation}
in $\Fun^{desc}(\Mf^{op},\Sp)$.
From Lemma \ref{loi1} we conclude:
\begin{kor}\label{klh12}
We have equivalences
$$U(BA_{\infty})\simeq U(BA_{\infty}^{\nabla})\simeq {\Sigma} \sp({A_{top}})\ {\text{ and }} \quad {S}
(BA_{\infty})\simeq {S}(BA_{\infty}^{\nabla})\simeq H(\iota(A^{\delta}[1]))\ .$$
\end{kor}

We again analyse the case of $U(1)$ further.
\begin{lem}\label{pqw1}
We have equivalences
$$BU(1)_{\infty}\simeq H(\Diff^{-1}(\Z[2]\to \R[2]))\ {\text{ and }} \quad  BU(1)_{\infty}^{\nabla}\simeq H(\Diff^{0}(\Z[2]\to \R[2]))\ .$$
\end{lem}
\proof
According to  Example \ref{hjw1}
the sheaf $\Diff^{-1}(\Z[2]\to \R[2])$  is defined by a pullback  diagram 
$$\xymatrix{
\Diff^{-1}(\Z[2]\to \R[2])\ar[r]\ar[d]& \iota((\sigma^{\ge 1}\Omega)[2])\ar[d]\\
\const(\iota(\Z[2]))\ar[r]&\iota(\Omega[2])\ .
}$$
Since we have a pull-back
$$\xymatrix{
 L(\iota(Y(U(1))[1]))\ar[d]\ar[r]^-{\frac{1}{2\pi i}d\log} & \iota((\sigma^{\ge 1}\Omega)[2])\ar[d]\\
\const(\iota(\Z[2]))\ar[r]&\iota(\Omega[2])
}$$
we conclude that
$$\Diff^{-1}(\Z[2]\to \R[2])\simeq L(\iota(Y(U(1))[1]))\ .$$
In analogy to  \eqref{wie1} we have for an abelian group
$A$  that \begin{equation}\label{pqwe1}\sp(\Nerve(A{\rightrightarrows} *))\simeq H(\iota(A[1]))\ ,\end{equation} where
$\Nerve(A{\rightrightarrows} *)\in \Comm\Group(\sSet[W^{-1}])$ is the nerve of the Picard groupoid $A{\rightrightarrows} *$.

Applying this objectwise to $A=Y(U(1))$ we get
$$L(H(\iota(Y(U(1))[1])))\simeq BU(1)_{\infty}\ ,$$
and therefore the desired equivalence. 
  
The second statement follows by analogous calculations.
Here we compare the definition of $\Diff^{0}(\Z[2]\to \R[2] )$ in  Example \ref{hjw1}
with the sheaf defined by the pull-back square
 $$\xymatrix{
L(\iota((Y(U(1))\stackrel{\frac{1}{2\pi i}d\log}{\longrightarrow} \Omega^1)[1]))\ar[r]\ar[d] & \iota((\sigma^{\ge 2}\Omega)[2])\ar[d]\\
\const(\iota(\Z[2]))\ar[r]&\iota((\Omega[2]))
}\ .$$ 
Using Lemma \ref{rrr6} we identify
$$L(H(\iota\big(Y(U(1))\stackrel{\frac{1}{2\pi i}d\log}{\longrightarrow} \Omega^1\big)[1]))\simeq BU(1)_{\infty}^{\nabla}\ .$$
\hB   

\subsection{The universal differential characteristic class}\label{secfive2}

Let now $G$ be a general Lie group. 
We define $$\Sigma^{\infty}_{+}BG_{\infty}:=L(\Sigma^{\infty}_{+} \bbB G) \quad \text{and} \quad \Sigma^{\infty}_{+}BG^{{\nabla}}_{\infty}:=L(\Sigma^{\infty}_{+} \bbB G^{{\nabla}})\ .$$ 
  
By construction, the differential  extensions $  \Sigma^{\infty}_{+}B G^{\nabla}_{\infty}$ and $\Sigma^{\infty}_{+} B G_{\infty}$  of  $\Sigma^{\infty}_{+}BG$ 
receive   the universal differential cohomology 
class of  $G$-bundles with and without connection. {From Lemma \ref{loi1} we get equivalences
\begin{equation}
 U(\Sigma^{\infty}_{+}BG^{\nabla}_{\infty})\simeq U(\Sigma^{\infty}_{+}BG_{\infty})\simeq \Sigma^{\infty}_{+} BG
\label{underlyingBGinfinity}
\end{equation}
and 
$$S(\Sigma^{\infty}_{+}B G^{\nabla}_{\infty})\simeq S(\Sigma^{\infty}_{+} B G_{\infty})\simeq
\Sigma^{\infty}_{+}BG^{\delta}\ .$$

\bigskip

As an application of our theory we reproduce the  construction of the  differential refinement of an integral characteristic classes first obtained by    Cheeger-Simons \cite{MR827262}.
  We consider an invariant polynomial $z\in I^{p}(\gaaa)^{G}$. It induces a map of {sheaves} of  sets
$$\Omega^{1}\otimes g\stackrel{\omega\mapsto d\omega+[\omega,\omega]}{\longrightarrow} \Omega^{2}\otimes \gaaa\stackrel{z}{\to} Z^{2p}(\Omega)\ .$$
This map is $Y(G)$-invariant, where $Y(G)$ acts trivially  {on} $Z^{2p}(\Omega)$.
We get an induced map of action groupoids
$$(Y(G)\times \Omega^{1}\otimes g{\rightrightarrows}\Omega^{1}\otimes g)\to (Y(G)\times Z^{2p}(\Omega){\rightrightarrows} Z^{2p}(\Omega))$$ and, by applying the nerve functor and sheafification,
$$\bbB G^{\nabla} \to     \bbB G\times  i(Z^{2p}(\Omega))\ ,$$
where $i$ is the map \eqref{vm10}.
We compose with the projection $\bbB G\times  i(Z^{2p}(\Omega))\to i(Z^{2p}(\Omega))$ and the equivalence
$i(Z^{2p}(\Omega))  \simeq  \Omega^{\infty}H(\iota(Z^{2p}(\Omega){[0]} ))$ in order to get the map
\begin{equation}\label{vm20}\bbB G^{\nabla}\to \Omega^{\infty}(H(\iota(Z^{2p}(\Omega){[0]} )))\ .\end{equation}
We apply the homotopification   to the adjoint of  \eqref{vm20}
  and get, using Lemma \ref{rrq1} {and Remark \ref{uuzzii10}}, a commutative diagram 
\begin{equation}\label{klo1}\xymatrix{   \Sigma^{\infty}_{+}B G_{\infty}^{\nabla}\ar[r]\ar[d]& H(\iota(Z^{2p}(\Omega ){[0]}))\ar[d]\\
\const( \Sigma^{\infty}_{+} BG)\ar[r]^{\const(\bz)}&\const(  H\iota(\R[2p]))
}\ .\end{equation}

The upper horizontal map is derived by the construction above from the choice of  the invariant polynomial $z\in I^{p}(\gaaa)^{G}$, and it determines the map $\bz:\Sigma^{\infty}_{+}BG\to H(\iota(\R[2p]))$ via homotopification. We have in particular reproduced the classical map
$I^{p}(\gaaa)^{G}\to H^{2p}(BG;\R)$.

We now assume that we have chosen an integral refinement of the class $\bz$, i.e. a lift
\begin{equation}\label{klo2}\xymatrix{&H\Z[2p]\ar[d]\\\Sigma_{+}^{\infty} BG\ar[r]^{\bz}\ar[ur]^{\bz^{\Z}}
&H\R[2p]}\ .\end{equation}
Then we get the two  squares below.
$$\xymatrix{  \Sigma^{\infty}_{+} B G_{\infty}^{\nabla} \ar@{-->}[r]^-{\hat \bz^{\Z}} \ar[d]&H(\Diff^{0}(\Z[2{p}]\to \R[2{p}]))\ar[r]\ar[d]&H(\iota({Z^{2p}(\Omega){[0]} ))} \ar[d]\\
\const(\Sigma^{\infty}_{+}BG)\ar[r]^{\bz^{\Z}}&\const(H(\iota(\Z[2p])))\ar[r]&\const(H(\iota(\R[2p])))
}$$
 
In order to connect with the usual notation we set $$\widehat{H\Z}^{2p}(M):=H^{0}(\Diff^{0}(\Z[2{p}]\to \R[2{p}])(M))\ .$$
{The  Cheeger-Simons construction \cite{MR827262} associates to a pair of  $z$ and its lift $\bz^{\Z}$}  a differential characteristic class for
$G$-bundles with connection $(P,\omega)$ on $M$
$$\hat z(P,\omega)\in \widehat{H\Z}^{2p}(M)\ .$$
In terms of the map
 $$\hat \bz^{\Z}:   \Sigma^{\infty}_{+}\bbB G_{\infty}^{\nabla}\to H(\Diff^{0}(\Z[2{p}]\to \R[2{p}]))$$ it is given by the induced map in $\pi_{0}$:
 $$\hat \bz^{\Z}([P,\omega])=\hat z(P,\omega)\ ,$$
 where we consider $[P,\omega]\in \pi_{0}(\Sigma^{\infty}_{+}\bbB G_{\infty}^{\nabla}(M))$.
 
 \bigskip
 
 \begin{ex}\label{heute200}{\rm
 {
One motivation to consider differential cohomology theories is the construction of secondary invariants. In the present set-up secondary invariant are obtained by an application of the functor $S$ to differential lifts of primary invariants.  If we apply $S$ to $\hat \bz^{\Z}$, then we get a map
$$ \Sigma^{\infty}_{+}BG^{\delta} \simeq  S(\Sigma^{\infty}_{+}B G^{\nabla}_{\infty})  \stackrel{S(\hat \bz^{\Z})}{\to} S(H(\Diff^{0}(\Z[2{p}]\to \R[2{p}])) \stackrel{Lemma \ref{rrq1}}{\simeq}  H(\iota(\R/\Z[2p-1])\ .  $$  Its homotopy class  is the classical   secondary  characteristic class $$\check{z}\in H^{2p-1}(BG^{\delta};\R/\Z)$$
for flat $G$-bundles associated to the pair  $(z,\bz^{\Z})$ already found by Cheeger-Simons \cite{MR827262}.
}}\end{ex}

\section{\texorpdfstring{$K$}{K}-theory}\label{ttz15}

In this section we will approach algebraic and topological $K$-theory via the  group completion functor which is defined as the left-adjoint in the adjunction 
   $$K:\Comm\Mon(\sSet[W^{-1}])\leftrightarrows \Comm\Grp(\sSet[W^{-1}]):\mathrm{inclusion}$$
{between commutative monoids and groups in the $(\infty,1)$-category of spaces.} We let $\Cat[W^{-1}]$ be the localization of the one-category of categories $\Cat$ at equivalences. Further let $\Nerve:\Cat\to \sSet$ be the functor which takes the nerve of a 1-category. 
Since it maps  equivalences between categories to homotopy equivalences of simplicial sets it induces a functor between $(\infty,1)$-categories (denoted by the same symbol)  $\Nerve:\Cat[W^{-1}]\to \sSet[W^{-1}]$ which extends the functor previously defined on groupoids.
 By $\Iso$ we denote the functor that maps a 1-category to
its maximal subgroupoid. 

\begin{ddd}
We define the functor
$$\cK:\Comm\Mon(\Cat[W^{-1}])\to \Sp$$
as the composition
$$\cK:=\sp\circ K\circ \Nerve\circ \Iso\ .$$
\end{ddd}

This definition is justified by the fact, that if $R$ is a commutative ring and $P(R)$ is the symmetric monoidal category of finitely generated projective $R$-modules with respect to the direct sum, then
$$KR:=\cK(P(R))$$ is a connective spectrum which represents Quillen's {algebraic }$K$-theory of $R$.

\bigskip

We let $$\Vect, \Vect^{\nabla}\in \Fun^{desc}(\Mf^{op},\Comm\Mon(\Cat[W^{-1}]))$$
be the 
{sheaves of symmetric monoidal categories}
 which associate  to a manifold $M$ the symmetric monoidal category of complex vector bundles (with connection) with the symmetric monoidal structure given by the direct sum. We  define the objects
$$\widehat \ku:=L(\cK(\Vect)) \quad \text{and} \quad \widehat{\ku}^{\nabla}:=L(\cK(\Vect^{\nabla}))$$
 in  $\Fun^{desc}(\Mf^{op},\Sp)$.
Since the group completion functor is a left-adjoint it does not preserve sheaves in general. Therefore we must add the sheafification functor in the definition above.

\begin{rem}{\em 
The symmetric monoidal structure on these  {sheaves} given by the tensor product of vector bundles induces on $\widehat\ku$ and $\widehat \bku^{\nabla}$ the structure of sheaves with values in $\CAlg(\Sp)$ 
(i.e. $E_\infty$-ring spectra), see  \cite{GGN} for a discussion {in the setting of $(\infty,1)$-categories.}  
}\end{rem}

\begin{lem}\label{loi2}
We have equivalences
$${  U(\widehat \ku) \simeq {\ku} \quad \text{and} \quad   S(\widehat \ku) \simeq  K\C}\ .$$
\end{lem}
\proof
{
We start with the equivalence 
$$\Iso(\Vect)\simeq \bigsqcup_{n\ge 0} \Bun(GL(n,\C))$$ 
in $\Fun^{desc}(\Mf^{op},\CommMon(\Groupoids[W^{-1}]))$, where on the right-hand side we write a coproduct of sheaves of groupoids which has a natural commutative monoid structure.
Since the nerve functor $\Nerve$ preserves products we obtain the equivalence
$$\Nerve(\Iso(\Vect))\simeq \bigsqcup_{n\ge 0}  \bbB GL(n,\C)$$ 
in $\Fun^{desc}(\Mf^{op},\CommMon(\sSet[W^{-1}]))$.}   
{The formula $\cH=L\circ \bs$ shows that $\cH$ can be expressed in terms of  iterated
sifted colimits and limits. Since the forgetful functor from sheaves of commutative monoids
to sheaves of spaces commutes with sifted colimits {(as shown in \cite[Corollary 3.2.3.2.]{HA}) and limits} it follows that the following diagram commutes:
$$\xymatrix{\Fun^{desc}(\Mf^{op},\CommMon(\sSet[W^{-1}]))\ar[r]^-{\mathrm{forget}}\ar[d]^{{\cH^{{\CommMon}}}}&\Fun^{desc}(\Mf^{op},\sSet[W^{-1}]))\ar[d]\ar[d]^{{\cH}}\\
\Fun^{desc,h}(\Mf^{op},\CommMon(\sSet[W^{-1}]))\ar[r]^-{\mathrm{forget}}&\Fun^{desc,h}(\Mf^{op},\sSet[W^{-1}]))}\ .$$
 By   Lemma \ref{loi1}  we have  an  equivalence
\begin{equation}\label{tho3obla}
U(\Nerve(\Iso(\Vect)))\simeq   \bigsqcup_{n\ge 0} BGL(n,\C)\ 
\end{equation}
 in $ \sSet[W^{-1}]$.
If we use this equivalence   in order transport the commutative monoid structure on
$\cH^{\CommMon}(\Nerve(\Iso(\Vect)))(*)$, then we obtain
the classical commutative monoid structure on   $\bigsqcup_{n\ge 0}  BGL(n,\C)$ such that
$$ \sp(K(\bigsqcup_{n\ge 0} BGL(n,\C))) \simeq  \ku\ .$$ 

Since $\sp\circ K$ commutes with $\const$ and $\cH$ 
 we get
 $$\cH(\cK( \Vect))   \simeq  \const(\ku)\ .$$
 
 We get the required equivalence
 $\cH(\widehat \ku) \simeq   \const(\ku)$.
 
 In order to show the second assertion  we use the identity $\ev_{*}(L(F)) \simeq  F(*)$.
 Now $$\ev_{*}(\cK(\Vect)) \simeq \cK(P(\C)) {\simeq}  K\C\ .$$
\hB

Therefore $\widehat{\bku}$ is a differential extension of the spectrum $\bku$.
Moreover the map
$S(\widehat\bku)\to U(\widehat \bku)$ is the canonical map
$K\C\to \bku$ {from algebraic K-theory to complex K-theory} so that  $ {\Sigma^{-1}Z}(\widehat{\bku})$ is the relative  $K$-theory spectrum
$K^{rel}\C$.
It seems to be an interesting problem to understand the sheaf $\cA(\widehat{\bku})$ and $\cZ(\widehat{\bku})$. The latter is a pure differential extension of $\Sigma K^{rel}\C$.

 \bigskip

We now consider vector bundles with connections  and
define
$$\widehat{\bku}^{\nabla}:=  L(\cK(\Vect^{\nabla}))\in \Fun^{desc}(\Mf^{op},\Sp)\ .$$
We have a forgetful map $\Vect^{\nabla}\to \Vect$.

\begin{lem}
The map
$$\bs(\Nerve(\Iso(\Vect^{\nabla})))\to \bs(\Nerve(\Iso(\Vect)))$$
is an equivalence
\end{lem}
\proof
One checks directly that 
$$\Nerve(\Iso(\Vect^{\nabla}))[p](\Delta^{\bullet}\times M)\to \Nerve(\Iso(\Vect))[p](\Delta^{\bullet}\times M)$$
is a Kan fibration for every $p\in \nat$ and manifold $M$.
This implies the assertion.
\hB 

\begin{kor}
We have
$${U(\widehat{\ku}^{\nabla}) \simeq  \bku \quad \text{and} \quad S( \widehat{\bku}^{\nabla}) \simeq  K\C}$$
Consequently, $\widehat{\bku}^{\nabla}$ is a
differential refinement {of} $\bku$.
\end{kor}
 
The differential refinement $\widehat{\bku}^{\nabla}$ has the same underlying map
$K\C\to \bku$ as $\widehat{\bku}$. So $\widehat{\bku}^{\nabla}$ differs from $\widehat{\bku}$ by the choice of the pure differential refinement $\cZ(\widehat{\bku}^{\nabla})$ of $K^{rel}\C$, only.

\bigskip
\newcommand{\Lin}{\mathrm{Line}}

\begin{rem}{\rm
Let $\hat E$ be an  arbitrary sheaf of spectra.
An additive $\hat E$-valued characteristic differential cohomology class for complex vector bundles is  {by definition}
a map $$\Nerve(\Iso(\Vect))\to \Omega^{\infty} \hat E$$
in $\Fun^{desc}(\Mf,\Comm\Mon(\sSet[W^{-1}]))$. By construction, 
${\widehat\bku}$ receives the universal additive {characteristic} differential cohomology class for vector bundles. 
Similarly, $\widehat{\bku}^{\nabla}$ receives the universal additive  characteristic differential
cohomology class for vector bundles with connections. 
}
\end{rem}

\bigskip

\begin{ex}{\rm
We construct the differential first Chern class $$\hat c_1:\widehat{\bku}^{\nabla}\to B\C^\times{}_{\infty}^{\nabla}\ .$$
Let $$\Lin^\nabla \in  \Fun^{desc}(\Mf^{op},\Comm\Mon(\Cat[W^{-1}]))$$ denote the sheaf which associates to a manifold $M$ the symmetric monoidal category of complex line bundles with connection and the tensor product.
The highest exterior power is {a} symmetric monoidal functor 
\begin{equation}\label{uuzzii20}\Lambda^{\text{top}}: \Vect^\nabla\to \Lin^\nabla \ .\end{equation}
The associated line bundle construction yields an equivalence $$\Bun\C^\times{}^\nabla\stackrel{{\simeq}}{\to}  \Iso(\Lin^\nabla)$$  
in  $\Fun^{desc}(\Mf^{op},\Comm\Mon(\Groupoids[W^{-1}]))$.
We now apply the composition $L\circ \sp\circ \cK\circ \Nerve$ to \eqref{uuzzii20} in order to get a map {(see \eqref{uuzzii21} for notation)}
$$\hat c_1:\widehat{\bku}^{\nabla}\to B\C^\times{}_{\infty}^{\nabla} \ .$$
One can check similarly to Lemma \ref{pqw1} that
$$B\C^\times{}_{\infty}^{\nabla} \simeq \Diff^{{0}}(\Z[{2}]\to \C[{2}])\ .$$
By construction, $\pi_{0}(B\C^\times{}_{\infty}^{\nabla}(M))\cong \pi_{0}(\Lin^{\nabla}(M))$,
and the map
$$\pi_{0}(\hat c_{1}):{\widehat{\bku}^{\nabla,0}}\to \pi_{0}(\Lin^{\nabla}(M))$$ extracts from the differential $K$-theory class the isomorphism class of determinant bundle with connection.

Note that this construction goes through without the decoration $\nabla$ and yields a class
$$\widehat{\bku}\to B\C^{\times}{}_{\infty}
\simeq \Diff^{{-1}}(\Z[{2}]\to \C[{2}])\ .$$
If we apply the functor $S$, then using the calculations \ref{klh12} and \ref{loi2}
we obtain a regulator
$$K\C\to H(\iota(\C^{\times}[{1}]))\ . $$ 
In the first homotopy group, if we identify $\pi_{1}(K\C)\cong \C^{\times}$ and
$\pi_{1}(H(\iota(\C^{\times}[{1}])))\cong \C^{\times}$, this regulator is the identity map.

In the next example we construct the higher-degree versions of the regulator.
}
\end{ex}

\subsection{The Hopkins-Singer differential \texorpdfstring{$K$}{K}-theory}

The  usual Hopkins-Singer version of differential connective complex $K$-theory is the sheaf $$ \widehat{\bku}_{HS}:=\Diff^{0}(\bku,\C[b],c)\ ,$$ where $c:\bku\to H(\iota(\C[b]))$ is a complex version of the Chern character  and $b$ is a variable of degree $-2$ {(see Section  \ref{ttz12} for the notation and construction)}. Our goal  in  the present section is to construct a transformation $$\hat r:\widehat{\bku }^{\nabla}\to \widehat{\bku}_{HS} $$
 which we call the cycle map.  On $\pi_{0}$ this map    sends   isomorphism classes of complex vector bundles with connection to their differential $K$-theory classes. In view of the homotopy theoretic definition of $\widehat{\bku}_{HS}$ the construction of the cycle map from a geometric source is not obvious.  The main input of our construction of the cycle map is  
the representation of  the Chern character $c:\bku\to H(\iota(\C[b))$ as the homotopification of a transformation induced by characteristic forms.
But first we  {use Lemma \ref{loi13} to }make the structure of $\widehat{\bku}_{HS}$ explicit.

\begin{lem}\label{uuzzii801}$\ $ 
\begin{enumerate}
 \item  {$\cZ(\widehat{\bku}_{HS})\simeq \sigma^{{\geq 0}} \big(b\Omega_{\C}[b]\big)$}, where $\Omega_{\C}=\Omega\otimes_{\R} \C$ is the complexified de Rham complex.  
 \item $U(\widehat{\bku}_{HS})\simeq \bku$
 \item $S(\widehat{\bku}_{HS})\simeq \Fib(\bku\oplus H(\iota(\C[0])) \to H(\iota(\C[b])))$ 
 \item $Z(\widehat{\bku}_{HS})\simeq H(\iota(b\C[b]))$
 \item $\cA(\widehat{\bku}_{HS})\simeq \Sigma^{-1}H(\iota(\sigma^{\le -1}(\Omega_{\C}[b]))$
 \item $\phi \simeq \left(\bku\stackrel{\ch}{\to} H(\iota(\C[b]))\to H(\iota(b\C[b]))\right)$, where the second map is given by the projection along the constants.
 \item For $x\in \widehat{\bku}_{HS}^{0}(\Delta^{1}\times M)$ we have the homotopy formula (specialization of \eqref{heu300})
\begin{equation}\label{sam100}\partial_{1}^{*}x-\partial_{0}^{*}x=a(\int_{\Delta^{1}} \hat \cR_{HS}(x))\ ,\end{equation} where
  $\hat \cR_{HS}: \widehat{\bku}_{HS}^{0} \to \Diff^{0}(\C[b])^{0}=Z^{0}(\Omega_{\C}[b])$
\end{enumerate}
\end{lem}

We now approach the Chern character via characteristic forms. If $(V,\nabla)$ is a complex vector bundle with connection on a manifold $M$, then we consider the curvature
$R^{\nabla}\in \Omega^{2}(M,\End(V))$ and define the Chern character form
\begin{equation}\label{uuzzii6}\ch(V,\nabla):=\Tr \exp( bR^{\nabla})\in  Z^{0}(\Omega_{\C}[b])(M)\ .\end{equation}
We consider the abelian group $Z^{0}(\Omega_{\C}[b])(M)$ as a discrete symmetric monoidal category. 

By the additivity and naturality of the Chern character {form} we get a transformation
$$\ch:\Iso(\vect^{\nabla})\to Z^{0}(\Omega_{\C}[b])$$
 in $\Fun(\Mf^{op},\CommMon(\Cat[W^{-1}]))$.
 If $A$ is an abelian group considered as a discrete symmetric monoidal category, then by \eqref{wie1} we have an equivalence
 $$K(A)\simeq  H(\iota(A[0]))\ .$$
 Applying $L\circ K \circ \Nerve$ to the map $\ch$ we therefore get a map
 $$\widehat{\ch}:\widehat{\bku}^{\nabla}\to H(L(\iota( Z^{0}(\Omega_{\C}[b])[0])))\to H(\iota(\sigma^{\ge 0}(\Omega_{\C}[b]))){=H(\Diff^{0}(\C[b]))} \ .$$
 We now consider the diagram
 $$\xymatrix{
 \widehat{\bku}^{\nabla}\ar@/^1cm/[rr]^{\widehat{\ch}}\ar@{-->}[r]^{\hat r}\ar[d]&\widehat{\bku}_{HS}\ar[r]^-{\hat \cR_{HS}}\ar[d]&{H(\Diff^{0}(\C[b]))}\ar[d]\\
 \const(\bku)\ar@{=}[r]\ar@/_1cm/[rr]_{\cH(\widehat{\ch})}&\const(\bku)\ar[r]^-{\const(c)}&\const(H(\iota(\C[b])))
 }\ .$$
 The lower line is the result of an application of $\cH$ to the upper. Therefore  we get a natural filler of the outer square. By the definition of $\widehat{\bku}_{HS}$ in Example \ref{ttz12} the right square is a pull-back if we define $$c:=\ev_{*}(\cH(\widehat{\ch})):\bku\to H(\iota(\C[b]))\ .$$  The cycle map $\hat r:\widehat{\bku}^{\nabla}\to \widehat{\bku}_{HS}$ is now obtained from the universal property of this pull-back square.  On $\pi_{0}$ it induces a cycle map
 \begin{equation}\label{uuzzii40}\widehat{\bku}^{\nabla,0}(M)\to {\widehat{\bku}}_{HS}^{0}(M)\ , \quad [V,\nabla]\mapsto \hat r(V,\nabla)\end{equation}
 which maps a complex vector bundle with connection to its (Hopkins-Singer version) differential $K$-theory class. We have the relation
 $$\hat \cR_{HS}(\hat r(V,\nabla))={\widehat{\ch}}(V,\nabla)\ .$$
 If ${\tilde\nabla}$ is a connection on $\tilde V:=\pr^{*}V$ over $\Delta^{1}\times M$ connecting two connections
 $\nabla_{i}:=\partial_{i}^{*} \tilde \nabla$, $i=0,1$, then
 $$\tilde \ch(\nabla_{1},\nabla_{0}):=[\int_{\Delta^{1}} \ch(\tilde V,\tilde  \nabla)]\in \Omega_{\C}[b]^{-1}(M)/\im(d)$$ is called the transgression of the Chern character form.  It only depends on the connections $\nabla_{i}$, but not on the choice of the connecting family $\tilde \nabla$.
The homotopy formula \eqref{sam100} gives
\begin{equation}\label{sam101}\hat r(V,\nabla_{1})-\hat r(V,\nabla_{0})=a(\tilde \ch(\nabla_{1},\nabla_{0}))\ . \end{equation}
 
  }
 \bigskip

 \begin{ex}{\em 
 Here we discuss the construction of the regulator for the  algebraic $K$-theory of $\C$  
as a secondary invariant of  the cycle map $\hat r:\widehat{\bku}^{\nabla}\to \widehat{\bku}_{HS}$. To this end
 we  apply the functor $S$ {(evaluation at the point)} to the {cycle} map {$\hat r:\widehat{\bku}^{\nabla}\to \widehat{\bku}_{HS}$} and use the calculation {of} Lemma \ref{loi2}   in order to get  a map
 $$r:K\C\to  S(\widehat{\bku}_{HS})\ .$$
 In homotopy {it} induces the regulator maps 
 $$r_{2i-1}:\pi_{2i-1}(K\C)\to \pi_{2i-1} ( S(\widehat{\bku}_{HS}))\cong \C/\Z$$ for all $i\ge 1$, where we use {part 3.} of Lemma \ref{uuzzii801}  
 for the second isomorphism.
 In \cite{MR750694} an analogous
 map $r_{\Q}:K\C\to \Sigma^{-1} \bku \Q/\Z$ has been used to represent
 $\Sigma^{-1} \bku \Q/\Z$ as a summand of $K\C$. This has been employed in \cite{MR772065}
 in order to show that $r_{2i-1}$ induces an isomorphism 
 $$\pi_{2i-1}(K\C)_{tors}\stackrel{\sim}{\to} \Q/\Z\subset \C/\Z\ , \quad  i\in \nat\setminus \{0\}\ .$$
  }
 \end{ex}

\subsection{Differential loop \texorpdfstring{$K$}{K}-theory}
 
Recently  a differential refinement $K^{0}_{TWZ}$ of complex $K$-theory was proposed in   \cite{2012arXiv1201.4593T} whose differential form data is given by equivariant differential forms on loop spaces. Its construction will be recalled in Example \ref{sam200}. In the present example we introduce a sheaf of spectra $\widehat{\bku}_{loop}$ (see Equation \eqref{uuzzii5})  such that $\widehat{\bku}_{loop}^{0}$ could be considered as {the best}  approximation of $K^{0}_{TWZ}$ fitting in the framework of the present paper {(see the corresponding Remark \ref{remloop})}. We again construct {a} cycle map
$$\hat r_{loop}:\widehat{\bku}^{\nabla}\to \widehat{\bku}_{loop}\ .$$

Below in Example \ref{sam4001} we  show
that $\widehat{\bku}_{loop}^{0}$ captures more geometric information about a  vector bundle with connection  than $\widehat{\bku}^{0}_{HS}$.
Later in Example \ref{sam200}
  we construct a transformation
$$b:K^{0}_{TWZ}\to \widehat{\bku}_{loop}^{0}$$
{and} show that it is in general not injective.

  We start with the Bismut Chern character form of a complex vector bundle with connection $(V,\nabla)$ on a manifold $M$  
 $$\BCH(V,\nabla)\in \Omega_{\C}({LM})[b^{-1}][[b]]^{0}\ , \quad \quad (d+b^{-1}i_{\xi})\BCH(V,\nabla)=0\ , $$
 where we refer to 
  \cite{MR786574}, \cite{MR1113683}, \cite{MR2712304}, \cite{2012arXiv1201.4593T} for definitions and explicit formulas. We just note that its zero form component is given at a loop $\gamma\in LM$ by the trace of the holonomy
  \begin{equation}\label{heut10}\BCH(V,\nabla)^{0}(\gamma)=\Tr\:\hol(V,\nabla)(\gamma)\ .\end{equation}
    The Bismut Chern character {form} is a zero cycle in the complex $F^{\prime}(M)$ introduced in  subsection \ref{uuzzii30}. 
 On constant loops it restricts to the Chern character form $\ch(V,\nabla)$ given by \eqref{uuzzii6}. 
   By {the} naturality and additivity the Bismut Chern character form  provides  a transformation
$$\BCH: \Iso(\Vect^{\nabla})\to Z^{0}(F^{{\prime}})$$ in $\Fun(\Mf^{op},\CommMon(\Cat[W^{-1}]))$. Applying the composition $L\circ K\circ \Nerve $ we get the transformation
$$\widehat{\BCH}:\widehat{\bku}^{\nabla}\to H(L(\iota(Z^{0}(F^{{\prime}})[0])))\to H(\hat F^{{\prime}}_{loop})$$
(see \eqref{heute100} for notation)
such that the following diagram commutes
$$\xymatrix{\widehat{\bku}^{\nabla}\ar[r]^{\widehat{\BCH}}\ar[rd]^{\widehat{\ch}}&H( \hat F^{{\prime}}_{loop})\ar[d]\\&{H(\Diff^{0}(\C[b,b^{-1}]))}}\ ,
  $$ 
  where the vertical arrow is given by evaluation at constant loops. In particular we get an equivalence
  $U(\widehat{\BCH}) \simeq  U(\widehat{\ch})$ of maps between spectra $\bku\to H(\iota(\C[b,b^{-1}]))$. 
We define the sheaf $$\widehat{\bku}_{loop}\in \Fun^{desc}(\Mf^{op},\Sp)$$ as the pull-back
 \begin{equation}\label{uuzzii5}\xymatrix{\widehat{\bku}_{loop}\ar[rr]^{\hat \cR_{loop}}\ar[d]&&H( \hat F^{{\prime}}_{loop})\ar[d]\\
 \const(\bku)\ar[rr]^-{\cH(\widehat{\BCH})}&&\const(H(\iota(\C[b,b^{-1}])))
}\ ,\end{equation}
where the right vertical arrow is given by the unit of the homotopification, {which is determined in}  Lemma \ref{uuzzii2}. 
{Concerning the lower horizontal map we have an equivalence
$$\cH(\widehat{\BCH})\simeq\const\left(\bku\stackrel{c}{\to} H(\iota(\C[b]))\to H(\iota(\C[b,b^{-1}]))\right)\ ,$$ where the second map is the obvious inclusion.}
By construction, the sheaf
$\widehat{\bku}_{loop}$ is a differential refinement of $\bku$.
From Lemma \ref{uuzzii4} we get:
\begin{lem}
\begin{enumerate}
\item $\cZ(\widehat{\bku}_{loop})(M) \simeq H(\iota(\Cone( \cF^{0}L(F^{{\prime}}) \to L(\sigma^{\ge 0}(F^{{\prime}} )))))$
\item $U(\widehat{\bku}_{loop})\simeq \bku$.
\item $Z(\widehat{\bku}_{loop})\simeq H(\iota(b\C[b]))$ 
\item $S(\widehat{\bku}_{loop}) \simeq 
{\Fib\left(\bku\to H(\iota(b\C[b]))\right)}$ 
\item $\cA(\widehat{\bku}_{loop}) \simeq \Sigma^{-1}H(\iota(\sigma^{\le -1 }(L( F^{{\prime}}))))$   
\item $\phi \simeq \left(\bku\stackrel{\ch}{\to} H(\iota(\C[b]))\to H(\iota(b\C[b]))\right)$, where the second map is 
given by the projection along the constants.
\item The integration map is given by the same formula as  Lemma \ref{uuzzii4}, $7.$ 
\end{enumerate}
\end{lem}

\begin{rem}{\rm
We observe by comparison with Lemma \ref{uuzzii801} that this data coincides with that of $\widehat{\bku}_{HS}$ except for the $\cA$ and $\cZ$-parts. The $\cA$ and $\cZ$-parts of $\widehat{\bku}_{loop}$ and $\widehat{\bku}_{HS}$   are related by maps
  induced by the restriction of forms to constant loops.
Since their characteristic maps coincide, both differential cohomologies
$\widehat{\bku}_{HS}^{0}$ and $\widehat{\bku}_{loop}^{0}$ encode the same secondary
invariants.\bigskip
}\end{rem}

In order to define the cycle map  for loop  differential $K$-theory we extend the diagram \eqref{uuzzii5} to 
$$\xymatrix{\widehat{\bku}^{\nabla}\ar@{-->}[r]^{\hat r_{loop}}\ar@/^1cm/[rr]^{\widehat{\BCH}}\ar[d]&\widehat{\bku}_{loop}\ar[r]^{\hat \cR_{loop}}\ar[d]&H( \hat F^{{\prime}}_{loop})\ar[d]\\\const(\bku)\ar@{=}[r]\ar@/_1cm/[rr]_{\cH(\widehat{\BCH})}&
 \const(\bku)\ar[r]^-{\cH(\widehat{\BCH})}&\const({H}(\iota(\C[b,b^{-1}])))
}\ .$$
The outer square naturally commutes and provides the cycle map
$$ \hat r_{loop}:\widehat{\bku}^{\nabla}\to \widehat{\bku}_{loop}\ .$$
It induces the cycle map
$$\hat r_{loop}: \widehat{\bku}^{\nabla,0}(M) \to  \widehat{\bku}^{0}_{loop}(M) \ , \quad [V,\nabla]\mapsto \hat r_{loop}(V,\nabla)$$ by applying $\pi_{0}$.  \bigskip\\
 For further application we need a case of the homotopy formula which is the analog of \eqref{heu300} and involves the map 
 $$\hat \cR_{loop}:\widehat{\bku}_{loop}^{0 }\to \hat F_{loop}^{{\prime},0}\ .$$
We are going to employ the notation from Lemma  \ref{uuzzii4}, $7.$  
\begin{lem}\label{heut60}
If $x\in \widehat{\bku}^{0}_{loop}(\Delta^{1}\times M)$, then $\partial_{1}^{*}x-\partial_{0}^{*}x=\partial_{1}^{*}a(\bH(\hat \cR_{loop}(x)))$. 
\end{lem} 
\proof 
The proof is similar to the one of Lemma \ref{sam300}.
One just uses the integration map for $\hat F_{loop}^{\prime}$ (Lemma \ref{uuzzii4}, 7.) instead of the one for $\Diff^{m}(C)$. Since $R_{\widehat{\bku}_{loop}}(x)=R_{\hat F_{loop}^{{\prime}}}(\hat \cR_{loop}(x))$ we get

 $$\partial_{1}^{*}x-\partial_{0}^{*}x\stackrel{\eqref{ghb3}}{=}a(\int R_{\widehat{\bku}_{loop}}(x))\stackrel{Lemma \:\ref{uuzzii4}, 7.}{=} a(\partial_{1}^{*}\bH( \hat \cR_{loop}(x)))\ .$$
\hB

  If $\nabla,\nabla^{\prime}$ are two connections on a complex vector bundle $V$ on $M$, then we can connect them   by a connection $\tilde \nabla$ on the pull-back of $V$ to  $\Delta^{1}\times M$. We define the Bismut-Chern-Simons form by transgression of $\BCH$. In detail it is given in terms of the homotopy \eqref{heut401}  by
  \begin{equation}\label{heut91}\BCS(\nabla,\nabla^{\prime}):=\partial_{1}^{*} \bH(\BCH(\tilde V,\tilde \nabla))\in F^{{\prime},-1}(M)/\im(d)\ .\end{equation}
  We let $\ell$ denote the map which maps a section of a presheaf to its  image in the associated sheaf.
  From Lemma \ref{heut60} and using the relation $\cR_{loop}(\hat r_{loop}(\tilde V,\tilde \nabla))=\ell(\BCH(\tilde V,\tilde \nabla))$ we get the homotopy formula
 \begin{equation}\label{heut90}\hat r_{loop}(V,\nabla)-\hat r_{loop}(V,\nabla^{\prime})=a(\ell(\BCS(\nabla,\nabla^{\prime})))\ .\end{equation}
If one restricts this formula to constant loops, then one gets \eqref{sam101}. 

\begin{ex}\label{sam4001}{\rm
In the following example we show that
$\widehat{\BCH}$, and therefore $\widehat{\bku}_{loop}^{0}$, captures more geometric information of a bundle than
$\widehat{\ch}$ and the Hopkins-Singer type differential extension
$ \widehat{\bku}_{HS}^{0}$.
We consider the trivial bundle $\R^{2}\times \C^{2}\to \R^{2}$ with the connection $$\nabla:=d+s\left(\begin{array}{cc}0&-1\\1&0\end{array}\right)dt\ ,$$ where $(s,t)$ are the coordinates of $\R^{2}$. Its curvature is
$$R^{\nabla}=\left(\begin{array}{cc}0&-1\\1&0\end{array}\right) ds\wedge dt\ .$$
We see that ${\widehat\ch}(V,\nabla)=2$. In particular, $\widehat{\ch}$ does not distinguish this connection from the trivial connection $\nabla^{triv}$. Moreover,
the transgression Chern form between the two connections {vanishes. In order to see this
we take the connection $$\tilde \nabla:=d+su\left(\begin{array}{cc}0&-1\\1&0\end{array}\right)dt$$
on $ \R^{3}\times \C^{2}\to \R^{3}$, where  $(u,s,t)$ are the coordinates of the base.
It interpolates between the trivial connection $\nabla^{triv}$ at $u=0$ and the connection $\nabla$ at $u=1$. Its curvature is
$$R^{\tilde \nabla}= \left(\begin{array}{cc}0&-1\\1&0\end{array}\right) (uds\wedge dt+sdu\wedge dt)\ .$$
It follows that $\Tr\exp(bR^{\tilde \nabla})=2+b\Tr R=0$. This implies that $\tilde \ch(\nabla,\nabla^{triv})=0$.}

Using {the homotopy} formula \eqref{sam101} 
we see that  $$\hat r(V,\nabla)=\hat r(V,\nabla^{triv})$$ in
$  \widehat{\bku}_{HS}^{0}(\R^{2}) $  .\bigskip

  For a loop $\gamma\in L\R^{2}$ we let $A(\gamma)\in \R$ be the volume encircled by $\gamma$ and measured with respect to the volume form $ds\wedge dt$.
 From \eqref{heut10}  and an explicit calculation of the holonomy we get
$$\BCH(V,\nabla)^{0}(\gamma)=2\cos(A(\gamma))\ .$$ 
Therefore $$\BCH(V,\nabla)=2\cos(A(\dots))+O(b)\ .$$
The closed form $\BCH(V,\nabla)\in Z^{0}(F^{\prime}(L\R^{2}))$ induces an element  
$$\ell(\BCH(V,\nabla))\in H^{0}( \hat F^{\prime}_{loop}(\R^{2}))\ .$$ 
We  argue that its  image in the stalk at $0\in \R^{2}$ is different from the image of the  constant function with value $2$.
In fact, {the} image of the function $2\cos(A(\dots))$
in $ \colim_{0\in U\subseteq \R^{2}}C^{\infty}(LU)$ is different from the germ of the constant function
with value $2$. 
This shows that
$\widehat{\BCH}$ distinguishes $\nabla$ from $\nabla^{triv}$. It follows that
$$\hat r_{loop}(V,\nabla)\not=\hat r_{loop}(V,\nabla^{triv})$$ in $ \widehat{\bku}_{loop}^{0}(\R^{2}) $.
}
\end{ex}

\begin{ex}\label{sam200}{\rm
In this example we construct a map
$$b:\hat K^{0}_{TWZ}\to  \widehat{\bku}_{loop}^{0} \ ,$$
where $\hat K^{0}_{TWZ}$ denotes the loop differential $K$-theory introduced in \cite{2012arXiv1201.4593T}.  Recall from \cite{2012arXiv1201.4593T} that for a manifold $M$ the group $\hat K^{0}_{TWZ}(M)$ is defined by cycles {and} relations.
A cycle   is just a complex vector bundle with connection $(V,\nabla)$. Its class in $\hat K^{0}_{TWZ}(M)$ will be denoted by   $[V,\nabla]_{TWZ} $. We define  
$$b([V,\nabla]_{TWZ}):=\hat r_{loop}(V,\nabla)\ .$$ 

We must show that $b$ is well-defined. 
Two cycles $(V,\nabla)$ and $(V,\nabla^{\prime})$ (with the same underlying bundle) represent the same class in
$ \hat K^{0}_{TWZ}(M)$ if
the Bismut-Chern-Simons form $\BCS(\nabla,\nabla^{\prime})$ defined in \eqref{heut91} vanishes.
By the homotopy formula \eqref{heut90} we get
$\hat r_{loop}(V,\nabla)=\hat r_{loop}(V,\nabla^{\prime})$ as desired.

We now show that the map $b$ is not injective. This is related with the fact that $\ell$ has a kernel. We consider the bundle $(V,\nabla)$ on $S^{1}$ given by
$$S^{1}\times \C^{2} \to {S^{1}} \ , \quad  \nabla=d+ \left(\begin{array}{cc}0&-1\\1&0\end{array}\right)dt\ .$$

We claim that
$$\hat r_{loop}(V,\nabla)=\hat r_{loop}(V,\nabla^{triv})\ .$$
By the homotopy formula we know that $$\hat r_{loop}(V,\nabla)-\hat r_{loop}(V,\nabla^{triv})
=a(\ell(\BCS(\nabla,\nabla^{triv})))\ .$$ 
 We have $$d\ell(\BCS(\nabla,\nabla^{triv}))=\ell(\BCH(\nabla )-\BCH(\nabla^{triv}))=0 \in L(F^{\prime})^{0}(S^{1}) \ .$$  
By  Lemma \ref{uli1232} the cohomology class of $\ell(\BCS(\nabla,\nabla^{triv}))$ can be detected by restriction to the constant loops. The restriction of  $\BCS(\nabla,\nabla^{triv})$
is the form $\tilde \ch(\nabla,\nabla^{triv})\in Z^{-1}(\Omega_{\C}[b,b^{-1}](S^{1}))$. But this form vanishes. Therefore $\ell(\BCS(\nabla,\nabla^{triv}))$ is exact and the claim follows.
 In contrast, we have $\BCH(V,\nabla)^{0}(\gamma)=2\cos(1)\not=2$, where $\gamma=\id:S^{1}\to S^{1}$
(see  \cite[Sec. 6.1]{2012arXiv1201.4593T}). This implies that
$$[V,\nabla]_{TWZ}\not=[V,\nabla^{triv}]_{TWZ}$$
in $K_{TWZ}^{0}(M)$.

}\end{ex}

\subsection{Snaith type differential \texorpdfstring{$K$}{K}-theory}

In this section we want to discuss a sheaf of spectra which refines complex $K$-theory
and which is defined by a smooth version of the Snaith construction of $K$-theory.
\medskip

Recall from Section \ref{suspension} the geometric suspension construction 
$I_{\widetilde{S^n}}F$ for a sheaf $F \in \Fun^{desc}(\Mfo, \bC)$. It follows from equation \eqref{t19okt} by evaluation at the point that $$U(I_{\widetilde{S^n}}F) \simeq U(F)^{(S^n_{top},1)}\ ,$$ where  the right hand side denotes the reduced power object in the stable, presentable $\infty$-category $\bC$. For a sheaf
$F \in \Fun^{desc}(\Mfo, \bC)$  we define the object
\begin{equation}\label{ewhfoqweife} \widetilde F (S^n) := \fib \big(F(S^n) \to F\{1\}\big) = I_{\widetilde{S^n}}{F}(*)\end{equation} of $\bC$.
Instead of taking the fibre of the evaluation at the point $1\in S^{n}$ as in the definition of  $I_{\widetilde{S^n}}F$ we can alternatively consider the fibre of the map which takes the germ at $1$.
\begin{ddd}
We define
$$ I_{\overline{S^n}}F := \colim_{1 \in U\subset S^n}\Big(\fib(I_{S^n}F \to I_U F)\Big)$$
where the colimit runs over {the partially ordered set of all open neighbourhoods $U$}  of $1 \in S^n$. In analogy to \eqref{ewhfoqweife} we define  $\overline F (S^n) := I_{\overline{S^n}}F(*)$.
\end{ddd}
\begin{rem}\label{ufihweifewf}$ ${\rm
\begin{enumerate}\item
There are canonical restriction morphisms 
$I_{\overline{S^n}}F \to I_{\widetilde{S^n}}F \to I_{{S^n}}F$.
\item
The object $\overline F (S^n)$ is equivalent to the evaluation $F_c(\R^n)$ with compact support, discussed in \cite[Definition 4.151.]{skript}, and a similar description can be given for the sheaf $I_{\overline{S^n}}F$.
\item In the definition of $I_{\overline{S^{n}}}$ we could alternatively take the colimit over all closed codimension zero submanifolds of $S^{n}$ which contain $1$ in their interior.
\end{enumerate}}
\end{rem}
\begin{lem}\label{underlzing}
We have equivalences $$U(I_{\overline{S^n}}F) \simeq U(F)^{(S^n_{top},1)} \qquad
\text{and} \qquad
I_{\overline{S^n}} \circ I_{\overline{S^m}} \simeq I_{\overline{S^{n+m}}}\ .$$
\end{lem} 
\begin{proof}
For the first equivalence we use the fact that $U$ commutes with colimits and finite limits to get the equivalence
$$ U(I_{\overline{S^n}}) \simeq \colim_{1 \in U\subset S^n}\fib\Big(U(I_{S^n}F) \to U(I_U F)\Big) $$
We can of course restrict this colimit to a colimit over contractible open subsets $U$. Then, using the equivalences
$$\fib\big(U(I_{S^n}F) \to U(I_U F)\big) \simeq \fib\big( U(F)^{S^n_{top}} \to U(F)^{U_{top}}\big)
\simeq U(F)^{(S^n_{top},1)}$$
where we have used Lemma \ref{thop} for the first equivalence, we can identify our colimit with a colimit over a constant diagram. The claim now follows.
\medskip

We now prove the second assertion. 
We start with inserting the definitions and use  Remark \ref{ufihweifewf}, 3. 
\begin{eqnarray*}\lefteqn{
I_{\overline{S^m}}I_{\overline{S^n}} F \simeq}&&\\&&
\colim_{K \subset S^m} \Fib(I_{S^{m}}(\colim_{L\subset S^{n}} \Fib(I_{S^{n}}F\to I_{L}F))\to I_{K}(\colim_{L\subset S^{n}} \Fib(I_{S^{n}}F\to I_{L}F)))\ ,\end{eqnarray*} where
$L$ and $K$ run over the closed submanifolds of $S^{n}$ and $S^{m}$, respectively, which contain the base point in their interior. We now apply Lemma \ref{thointer} in order to commute the inner colimits with the insertions. We further combine the resulting two colimits in the argument of the outer $\Fib$ to one colimit and 
{use} the  stability of $\bC$  in order commute $\Fib$ with this colimit. We get the equivalence
 $$I_{\overline{S^m}}I_{\overline{S^n}} F \simeq \colim_{K \subset S^m,L\subset S^{n}} 
 \Fib(I_{S^{m}}\Fib(I_{S^{n}}F\to I_{L}F)\to I_{K}\Fib(I_{S^{n}}F\to I_{L}F))\ .$$ The usual identification
 $\R^{n}\cong S^{n}\setminus \{1\}$ induces an excision equivalence
 $$\Fib(I_{S^{n}}F\to I_{L}F)\cong \Fib(I_{\R^{n}}F\to I_{L^{\prime}}F)\ ,$$ where
 $L^{\prime}:=L\setminus \{1\}$ is a closed submanifold of $\R^{n}$.
 Using excision
 we therefore get the equivalence
  $$I_{\overline{S^m}}I_{\overline{S^n}} F \simeq \colim_{K^{\prime} \subset \R^{m},L^{\prime}\subset \R^{n}} 
 \Fib(I_{\R^{m}}\Fib(I_{\R^{n}}F\to I_{L^{\prime}}F)\to I_{K^{\prime}}\Fib(I_{S^{n}}F\to I_{L^{\prime}}F))\ ,$$
 where now the colimit runs over closed submanifolds of the euclidean spaces which are neighbourhoods of $\infty$. Since insertion commutes with $\Fib$ we  get
 \begin{equation*}
I_{\overline{S^m}}I_{\overline{S^n}} F \simeq
 \colim_{K^{\prime} \subset \R^{m},L^{\prime}\subset \R^{n}} 
\fib\Big( \fib\big(I_{\R^m \times \R^n}F \to I_{\R^m \times L^\prime}F\big) \to \fib\big(I_{K^\prime \times \R^n}F \to I_{K^\prime\times L^\prime}F\big)\Big)
\end{equation*}
We now observe that we can equivalently take the colimit over all open neighbourhoods of $\infty$ for which we use the notation $K^{c}$ and $L^{c}$.
To analyse the diagram over which the colimit is taken we denote it by 
$$ 
P := \fib\Big( \fib\big(I_{\R^m \times \R^n}F \to I_{\R^m \times L^c} F\big) \to \fib\big(I_{K^c\times \R^n}F \to I_{K^c\times L^c}F\big)\Big)
$$
We can describe $P$ equivalently as the fibre of the morphism
$$ I_{\R^m \times \R^n}F \to \big(I_{K^c \times \R^n}F\big) \times_{(I_{K^c \times L^c}F)} \big(I_{\R^m \times L^c}F\big)$$
Using the descent property of $F$ we see that the right hand side is equivalent to 
the sheaf $ I_{(K^c \times \R^n) \cup (\R^m \times L^c)} F.$
Putting everything together we get an equivalence
\begin{align*}
I_{\overline{S^m}}I_{\overline{S^n}} F &\simeq 
\colim_{K^{c} \subset \R^m, L^{c} \subset \R^n}
\fib\Big(I_{\R^m \times \R^n}F \to  I_{(K^c \times \R^n) \cup (\R^m \times L^c)} F\Big)\\
&\simeq 
\colim_{M^{c} \subset \R^{m+n}} \fib\Big(I_{\R^{m+n} }F \to  I_{M^{c} } F\Big) 
\simeq I_{\overline{S^{n+m}}} F\ ,
\end{align*}
where the last colimit runs over  of the partially ordered set of open neighbourhoods of $\infty$ of $\R^{n+m}$. In the second equivalence we have used  that the intersections $(K^c \times \R^n) \cup (\R^m \times L^c)$ are cofinal among those neighborhoods.
\end{proof}

Now assume that $F \in \Fun^{desc}(\Mfo, \Sp)$ has a refinement $F^r \in \Fun\big(\Mfo,\CAlg(\Sp)\big)$ to a sheaf of commutative ring spectra, and that we have {chosen} a class 
in $x \in \pi_m (\overline F(S^n))$.
If we choose a representative of $x$, then the multiplication induces a 
morphism
$$ F \to \Sigma^{-m}I_{\overline{S^n}}F\ , $$
and therefore, by iteration, a tower
\begin{equation}\label{cotower}
F \to \Sigma^{-m}I_{\overline{S^n}}F \to \Sigma^{-2m}I_{\overline{S^{2n}}}F \to \Sigma^{-3m}I_{\overline{S^{3n}}}F \to ... \ .
\end{equation}

\begin{ddd}\label{uin}
The sheaf $F[x^{-1}] \in \Fun^{desc}(\Mfo, \Sp)$ is defined as the colimit over the diagram \eqref{cotower}.
\end{ddd}
The diagram \eqref{cotower} depends on the choice of the representative of $x$. However different choices lead to equivalent diagrams and therefore to equivalent colimits.  

\begin{rem}{\rm
There is a variant of the construction where $x$ is not in $\pi_m (\overline F(S^n))$ but in 
$ \pi_m (\widetilde F(S^n))$. Then the tower takes the form
$$F \to \Sigma^{-m}I_{\widetilde{S^n}}F \to \Sigma^{-2m}I_{\widetilde{S^{n}}}^2 F \to \Sigma^{-3m}I_{\widetilde{S^{n}}}^3 F \to ...$$
We denote this variant by $F[x^{-1}]_0$. For $x \in \pi_m (\overline F(S^n))$ it is understood that we first restrict $x$ to $ \pi_m (\widetilde F(S^n))$ if we write $F[x^{-1}]_0$. eIn this case we have a canonical map $F[x^{-1}]\to F[x^{-1}]_{0}$.}
\end{rem}

\begin{lem}\label{leminversion}
In the situation of Definition \ref{uin} we have
\begin{enumerate}
\item
The spectrum $U(F)$ admits the structure of a commutative ring spectrum coming from the ring structure on $F^r$.
\item
The element $x \in \pi_m(\overline F(S^n))$ induces an element 
$U(x) \in \pi_{m+n} (U(F))$.
\item
We have equivalences of spectra $U(F[x^{-1}]) 
\simeq U(F[x^{-1}]_0)
 \simeq U(F)[U(x)^{-1}]$. \end{enumerate}
\end{lem}
\begin{proof}
For the first assertion we compare the two functors
$$U_{\Sp}: \Fun^{desc}(\Mfo, \Sp) \to \Sp$$
and 
$$
\qquad U_{\CAlg(\Sp)}:\Fun^{desc}(\Mfo, \CAlg(\Sp)) \to \CAlg(\Sp)$$ 
where $\CAlg(\Sp)$ denotes the presentable $(\infty,1)$ category of commutative ring spectra. We have the general formula which represents the  functor $U$ as $U(F) \simeq \colim_{\Delta^{op}} F(\Delta^n)$. Thus we can use the fact that the forgetful functor from $E: \CAlg(\Sp) \to \Sp$ commutes with sifted colimits in order to deduce  the equivalence 
$U_{\Sp}(F) \simeq E \circ U_{\CAlg(Sp)}(F^r)$. This shows the first claim. 

\medskip

It follows from Lemma \ref{underlzing} that we have 
$U(\Sigma^{-m}I_{\overline{S^n}}F) \simeq \Sigma^{-m-n}U(F)$.
The element $U(x) \in \pi_{n+m}(U(F)) = \pi_0(\Sigma^{-m-n}U(F))$ is then obtained by functoriality of $U$. The third assertion of the lemma follows from this since $U$ commutes with colimits.
\end{proof}

We  now use the abstract construction explained above to give another differential refinement of the periodic complex $K$-theory spectrum $KU$. Recall the definition of the sheaves of spectra 
$$ \Sigma^{\infty}_{+}BU(1)_{\infty} \quad \text{and} \quad \Sigma^{\infty}_{+}BU(1)^\nabla_{\infty}$$
from Section \ref{secfive2}.

\begin{rem}{\rm
Note that for every sheaf $F \in \Fun^{desc}(\Mfo, \sSet[W^{-1}])$ there is a canonical morphism
$ F \to \Omega^\infty L(\Sigma^\infty_+ F)$.
If $F$ is a sheaf of commutative group objects (with unit denoted by 1), then we have a morphism
of pointed objects
$$ (F,1) \longrightarrow (\Omega^\infty L(\Sigma^\infty_+ F),1) \stackrel{-1}{\longrightarrow} 
(\Omega^\infty L(\Sigma^\infty_+ F),0)$$
where $0$ is the canonical basepoint in $\Omega^\infty L(\Sigma^\infty_+ F)$ which comes from the infinite loop space structure. 
This induces for every manifold $M$ a morphism
$$ \overline F(M) \to \overline{\Omega^\infty L(\Sigma^\infty_+ F)}(M) \simeq \Omega^\infty\big(\overline{L(\Sigma^\infty_+ F)}(M)\big)$$
where the last equivalence follows from the definition if we use that $\Omega^\infty$ commutes with filtered colimits and taking fibres.
We get  induced maps
$$ \quad\pi_m(\overline F(S^n) )\to \pi_m(\overline{L(\Sigma^{\infty}_{+}F)}(S^n))\ .$$
which we denote {by abuse of notation as} $[x] \mapsto [x]-1$. Note that these maps are in general not additive. 
}
\end{rem}

Let $L \to S^2$ be the Hopf bundle considered as a smooth principal} $U(1)$-bundle. {It represents   a class in $\pi_0(\bbB U(1)(S^2))$. One can refine 
this class to a class  denoted $[L] \in \pi_0(\overline{\bbB U(1)}(S^2))$} by choosing a section in a neighborhood of $1 \in S^2$. Since any two germs of sections are related by a bundle isomorphism the class $[L]$ is unique. Thus we get a class 
$$b := [L]-1 \in \pi_0(\overline{\Sigma^{\infty}_{+}BU(1)_{\infty}}(S^2)) \ .$$

\newcommand{\KU}{\mathrm{KU}}
\begin{ddd}
We define the sheaves of spectra
$$ \widehat\KU_{Snaith} := \Sigma^{\infty}_{+}BU(1)_{\infty}[b^{-1}] \qquad 
{\text{ and }}
\qquad \widehat\KU_{Snaith,0} := \Sigma^{\infty}_{+}BU(1)_{\infty}[b^{-1}]_0\ . $$
\end{ddd}
\begin{prop}
The sheaves of spectra $\widehat\KU_{Snaith}$ {and $\widehat\KU_{Snaith,0}$} are refinements of periodic $K$-theory $\KU$, i.e. we have equivalences
$U(\widehat\KU_{Snaith}) \simeq \KU$  and $U(\widehat\KU_{Snaith,0}) \simeq \KU$.
\end{prop}
\begin{proof} We prove the first equivalence, the second works the same way. Since $U$ is a left adjoint it commutes with colimits. Using Lemma \ref{underlzing} and formula \eqref{underlyingBGinfinity} together with Lemma \ref{leminversion} for the multiplicative structure we get 
$$
U(\widehat\KU_{Snaith}) \simeq \colim \left( \Sigma^{\infty}_{+}BU(1) \stackrel{\cdot b}{\to} \Sigma^{\infty-2}_{+}BU(1) \stackrel{\cdot b}{\to} \Sigma^{\infty-4}_{+}BU(1) \stackrel{\cdot b}{\to} \ldots
\right) \ .
$$
But this colimit is equivalent to $\KU$ by Snaith's theorem \cite[{Theorem 9.1.1}]{snaith}.
\end{proof}

Let us close this chapter discussing a variant of the spectrum $\widehat\KU_{Snaith}$ with connection. Let $L \to S^2$ be the Hopf bundle with its canonical connection $U(2)$-equivariant connection  $\nabla$.   The curvature of this connection is the standard volume form on $S^2$. 
Since this form does not vanish in a neighbourhood of $1\in S^{2}$ we {cannot} lift $[L,\nabla] \in \pi_0(\bbB U(1)^\nabla(S^2))$ to $\pi_0(\overline{\bbB U(1)^{\nabla}}(S^2))$, but we have a lift to  $  \pi_0(\widetilde{\bbB U(1)^{{\nabla}}}(S^2))$. Then we get the class 
$$\hat b := [L,\nabla]-1 \in \pi_0(\widetilde{\Sigma^{\infty}_{+}BU(1)^\nabla_{\infty}}(S^2))$$

\begin{ddd}
We define the sheaf of spectra
$$ \widehat\KU^\nabla_{Snaith,0} := \Sigma^{\infty}_{+}BU(1)^\nabla_{\infty}[\hat b^{-1}] $$
\end{ddd} 
The same proof as above shows the following.
\begin{prop}
The sheaf of spectra $\widehat\KU^\nabla_{Snaith,0}$ is a refinement of periodic $K$-theory $\KU$, i.e. we have an equivalence
$U(\widehat\KU_{Snaith,0}) \simeq \KU$. Furthermore, there is a canonical morphism $\widehat\KU^\nabla_{Snaith,0} \to \widehat\KU_{Snaith,0}$. 
\end{prop}

\begin{rem}{\rm 
It is an interesting problem to define maps
$$\widehat\bku\to \widehat \KU_{Snaith}\ , \qquad \widehat\bku^{\nabla}\to \widehat \KU^{\nabla}_{Snaith,0} \ . $$
Equivalently, we are looking for additive differential Snaith-$K$-theory valued characteristic classes for vector bundles with and without connection. We plan to come back to this problem in future work. 
}\end{rem}

\section{Technical results about sheaves}\label{gfd1}

In this section we assume that
   $\bC$ is a presentable $(\infty,1)$-category. {We first want to provide a criterion to detect equivalences of $\bC$-valued sheaves.} 
   For $n\in \nat$ and $r\in \R^{>{0}}$ we let $K^{n}(0,r)\subseteq \R^{n}$ be the ball of dimension $r$.
\begin{ddd}
For $F\in \Fun(\Mf^{op},\bC)$ and $n\ge 0$ we define the stalk {$F(\R^{n}_{0})\in \bC$ by}
$$F(\R^{n}_{0} ):=\colim_{r\to 0} F(K^{n}(0,r))\ .$$
 \end{ddd}
  
We consider a map $f:A\to B$ between two $\bC$-valued presheaves $A,B\in \Fun(\Mf^{op},\bC)$.

\begin{lem} \label{lll24} 
If  for every $n\in \nat$   the  map {induced by $f$}
$$ A(\R^{n}_{0})\to B(\R^{n}_{0})$$
is an equivalence, then
$L(f):L(A)\to L(B)$ is an equivalence. \end{lem}
For a proof we refer to \cite[Lemma 4.11]{skript}. 

\bigskip
In the stable case, equivalences between $\bC$-valued sheaves can be tested on sequences of global objects like sequences of spheres or tori as long as the dimensions turn to infinity.    
\begin{lem}\label{test_equi}
Assume {that} $\bC$ is stable. Then a morphism $A \to B$ in $\Fun^{desc}(\Mf^{op},\bC)$ is an equivalence if and only if 
there exists a  sequence of manifolds $\{M_k\}_{k \in \nat}$ with the property  {$\sup\{\dim M_k\:|\:k\in \nat\} = \infty$} such that the evaluations $A(M_k) \to B(M_k)$ are equivalences.
\end{lem}
\proof 

 As a first step
we show  that $F(M)\simeq 0$ for one particular non-empty manifold $M$ with $\dim(M)=n$ implies that $F$  is trivial on  all manifolds of dimension $n$. Let $B^{\prime}\to B$ be an inclusion of open $n$-dimensional balls such that the closure of $B^{\prime}$ is still contained in $B$. Then there exists an embedding
$B\hookrightarrow M$ and
we can extend the inclusion $B^{\prime}\hookrightarrow  B$  to a smooth map $f: M \to B$. In this way we obtain a factorization of the restriction $F(B)\to F(B^{\prime})$  as $$ F(B) \to F(M) \to F(B')\ .$$ This composition is trivial since   $F(M)$ is trivial by assumption.

Let now $N$ be an $n$-dimensional manifold.
 We can find an open covering $\cB^{\prime}=(B_{i}^{\prime})_{i\in I}$ of $N$ such that all iterated intersections
$B_{i_{1}}^{\prime}\cap \dots \cap B_{i_{k}}^{\prime}$ are either empty or balls (such a covering is called good) and a covering
$\cB=(B_{i})_{i\in I}$ by larger balls such that $\overline{B_{i}^{\prime}}\subset B_{i}$ for all $i\in I$ with the same property. Then $\cB^{\prime}$ is a refinement of $\cB$. Using that $F$ is a sheaf  we have a diagram
$$\xymatrix{F(N)\ar[r]^-\simeq\ar@{=}[d]&\lim_{\Delta} F(\cB^{\prime,\bullet})\ar[d]\\
F(N)\ar[r]^-\simeq&\lim_{\Delta}F(\cB^{\bullet})}\ ,$$
where $\cB^{\bullet}$ and $\cB^{\prime,\bullet}$ denote the \v{C}ech nerves.
From the observation made above we conclude that the right vertical map vanishes.

As a second step we show that $F(M)\simeq 0$ implies that $F$  is also trivial on all manifolds $N$ with $\dim(N)\le \dim(M)$. In fact we can write such a manifold $N$ as a retract of $N\times \R^{\dim(M)-\dim(N)}$.
 By the first step we know $F(N\times \R^{\dim(M)-\dim(N)}) \simeq 0$, hence $F(N) \simeq 0$.
 
In order to finish the proof of the Lemma we observe that  a morphism $A \to B$ is an equivalence if and only if its fibre is trivial. But the fibre has the property that it vanishes on the manifolds $M_k$ for all $k\in \nat$, hence on all manifolds.
 \hB

For a manifold $S$ we  define the endofunctor
$I_{S}$ of $\Fun^{desc}(\Mf^{op},\bC)$  in the natural way such that
\begin{equation}\label{heut1000}I_{S}(\hat E)(M):=\hat E(S\times M)\ .\end{equation}
\begin{rem} \label{rem_power} {\rm The functor $I_{S}$ is equivalent to the power operation {with the representable sheaf} $y(S)$.
Indeed, since $\bC$ is presentable the category $\Fun^{desc}(\Mfo,\bC)$ is tensored and powered over the sheaf category
$\Fun^{desc}(\Mfo,\sSet[W^{-1}])$. For a manifold $S\in \Mf$ we have a canonical equivalence of functors
\begin{equation}
\label{natnatnat1}I_{S}(-){\simeq} (-)^{y(S)}:\Fun^{desc}(\Mfo,\bC)\to \Fun^{desc}(\Mfo,\bC)\ .
\end{equation}
This can most easily be checked by evaluating both sides on a manifold $M$. Let $G\in \bC$ and $F\in \Fun^{desc}(\Mfo,\bC)$  arbitrary. Then we have the chain of equivalences
\begin{eqnarray*}
\Map(G,F^{y(S)}(M))&\simeq& \Map(\const(G)\otimes y(M),F^{y(S)}) \simeq \Map(\const(G)\otimes y(M)\otimes y(S),F)\\& \simeq & \Map(\const(G)\otimes y(M\times S),F) \simeq \Map(G,F(S\times M))\ ,
\end{eqnarray*}
where the first and last equivalences are given {by} the definition of the tensor operation, the second equivalence 
is the universal property of the power operation, and the third equivalence uses the fact that the Yoneda embedding $y$ preserves products.

Furthermore, for  $X\in \sSet[W^{-1}]$ we have an equivalence
 \begin{equation}\label{zusatz1}(-)^{X}\simeq (-)^{\const(X)}\ .\end{equation}
If $S$ is a manifold, then we have a map $y(S)\to \cH(y(S))\simeq \const(S_{top})$. By the functoriality of the power operation, {\eqref{natnatnat1}} and \eqref{zusatz1} it induces a transformation
\begin{equation}\label{natnatnat}(-)^{S_{top}}\to {I_S}\end{equation}
 between functors on $\Fun^{desc}(\Mfo,\bC)$.}
\end{rem}

\bigskip

Our next task is to provide a formula for the homotopification $\cH$.
We let $\Delta^{n}\in \Mf$ be the standard simplex. The standard simplices combine to a give a cosimplicial manifold
$\Delta^{\bullet}:\Delta\to \Mf$.
{Recall that $\bC$ is a presentable $(\infty,1)$-category}.  
We define an endofunctor \begin{equation}\label{poi1}\bs:={\colim_{\Delta^{op}}\circ I_{\Delta^{\bullet}}}:\Fun(\Mf^{op},\bC)\to \Fun(\Mf^{op},\bC)\ .\end{equation} 
{More concretely we have ${{\bs}(F)(M)} {\simeq} \colim_{\Delta^{op}}F(\Delta^{op}\times M).$}
Recall the definition \eqref{hdef123} of the homotopficiation functor
$$ \cH^{pre} :\Fun(\Mf^{op},\bC) \to \Fun^h(\Mf^{op},\bC)\ .$$
{The following Lemma provides a formula for $\cH^{pre}$.}

\begin{lem}\label{lll25}$\ $ 
\begin{enumerate}
\item
If $F\in \Fun(\Mf^{op},\bC)$, then $\bs(F)$ is homotopy invariant.
\item
We have an equivalence  of functors $\cH^{pre}\simeq \bs$. 
\end{enumerate}
\end{lem}
\begin{proof}
We first show that  $\bs(F)$ is homotopy invariant. We use the following observations:
\begin{itemize}
\item
The functor $\bs: \Fun(\Mf^{op},\bC) \to \Fun(\Mf^{op},\bC)$ preserves colimits. 
\item
The category of homotopy invariant presheaves is closed under colimits.
\item
The category $\Fun(\Mf^{op},\bC)$ is generated under colimits by objects of the form
$y^{pre}(M) \otimes C \in \Fun(\Mf^{op},\bC)$ 
for all manifolds $M$ and objects $C$ of $\bC$,  where the tensor is the objectwise application of the natural tensor 
$\otimes:  \sSet[W^{-1}] \times \bC \to \bC$, and $y^{pre}(M)$ is the  {representable} sheaf $y(M)$ introduced 
in \eqref{sdf1} but considered as a presheaf.
 \end{itemize}
Combining these three observations, we see that it suffices to show that $\bs({y^{pre}(M)} \otimes C)$ is homotopy invariant. One can check by inserting definitions that
$$ \bs(y^{pre}(M) \otimes C) \simeq \bs(y^{pre}(M)) \otimes C\ .$$ 
It therefore suffices to show that $\bs(y^{pre}(M))$ is homotopy invariant. This is basically the statement that the smooth singular complex of a manifold is homotopy invariant, which is clearly true. 
In detail the argument is as follows.
 
\medskip

We
 will show that the projection $\Delta^{1}\times  {N}\to  {N}$ induces an equivalence
$$\bs\big(y^{pre}(M)\big)(N)\to \bs\big(y^{pre}(M)\big)\big(  {\Delta^{1}} \times N  ) = I\bs\big(y^{pre}(M)\big)\big(N )$$
for every manifold $N$. The inverse is given by $\partial_{0}^{*}$. Since
$\partial_{0}^{*}\circ \pr^{*}=\id_{\bs(y^{pre}(M))(N)}$ it suffices to show that
$\pr^{*}\circ \partial_{0}^{*}\simeq \id_{I\bs(y^{pre}(M))(N)}$.
\\
The smooth map $\Delta^{1}\times \Delta^{1}\to \Delta^{1}$, $(s,t)\to st$, induces a map
$Y(\Delta^{1})\to Y(\Delta^{1})\times Y(\Delta^{1})$. We
 use the rule
$I_{S}F=F^{Y(S)}$ for every manifold $S$ and  every sheaf $F\in \Fun(\Mf^{op},\Set)$ in order to construct  an internal homotopy
 \begin{equation}\label{chaos}IF=F^{Y(\Delta^{1})}\to F^{Y(\Delta^{1})\times Y(\Delta^{1})} = (IF)^{Y(\Delta^{1})}\ . \end{equation} We consider the product preserving  functor
$$\Sing^{s}:=\iota\circ \ev_{*}\circ I_{\Delta^{\bullet}} :\Fun(\Mf^{op},\Set)\to \sSet[W^{-1}]
\ .$$
It maps the adjoint $IF\times Y(\Delta^{1})\to IF$ of \eqref{chaos} to the map
$$\Sing^{s}(IF)\times \Sing^{s}(Y(\Delta^{1}))\to  \Sing^{s}(IF)\ .$$
We now specialize $F$ to  $F:=Y(M)^{Y(N)}$ and use the rule   
$$\Sing^{s}(I(Y(M)^{Y(N)})) \simeq  I\bs(y^{pre}(M))(N)$$  {in order to produce  
a map of simplicial sets
$$\Sing^{s}(I(Y(M)^{Y(N)}))\times  \Sing^{s}(Y(\Delta^{1}))\to  \Sing^{s}( I(Y(M)^{Y(N)})) \simeq   (I\bs(y^{pre}(M)))(N)$$
which restricts to  $\pr^{*}\circ \partial_{0}^{*}$ and $\id_{I\bs(y^{pre}(M))(N)}$
along the two inclusions $\Delta[0] \to \Sing^{s}(Y(\Delta))$. Next we use that 
$\Sing^{s}(Y(\Delta^{1}))$ is contractible 
(since it is the usual smooth singular complex of the contractible  manifold $\Delta^{1}$) to conclude that the two maps $\pr^{*}\circ \partial_{0}^{*}$ and $\id_{I\bs(y^{pre}(M))(N)}$ 
agree in the homotopy category of simplicial sets. In particular they are homotopic 
which concludes the proof.
}

\medskip

We now show $2.$
For a homotopy invariant presheaf $F$ the canonical map $F\to \bs(F)$ is an equivalence. Hence $1.$ implies that we have an equivalence 
  $\bs\stackrel{\sim}{\to} \bs^{2}$, i.e. $\bs$ is a localization of the category 
$\Fun(\Mf^{op},\bC)$. Thus there is a full subcategory $\Fun^\bs(\Mf^{op},\bC) \subset \Fun(\Mf^{op},\bC)$ of $\bs$-local objects and we have to show that it agrees with the homotopy invariant sheaves {(which are also a localization)}. The first part of the proof shows that
 $\Fun^\bs(\Mf^{op},\bC) \subseteq \Fun^h(\Mf^{op})$. The equivalence $\bs(F) \simeq F$ for a homotopy invariant presheaf shows {the other inclusion}, namely  $\Fun^h(\Mf^{op},\bC) \subseteq \Fun^\bs(\Mf^{op})$.
\end{proof}

Let us record the following formulas for the smooth singular complex of a manifold $M$:  \begin{equation}\label{dsf12}\sing(M) \simeq \Sing^{s}(Y(M))  \simeq \bs(y^{pre}(M))(*)\simeq \bs(y(M))(*)\ . \end{equation}
In the last term we use the convention also adopted below that if we apply $\bs$  to a sheaf, then we consider the latter as a presheaf   and consider the result   as a presheaf, too.  The functor $\bs$ preserves descent to some extend and we
  refer to Proposition \ref{propcom} for more details.
As a consequence of
 Proposition \ref{propo} the   homotopification functor 
$$ \cH :\Fun^{desc}(\Mf^{op},\bC) \to \Fun^{{desc,h}}(\Mf^{op},\bC)$$
is given by the formula $\cH = L \circ \cH^{pre} \simeq L \circ \bs$.
 In particular, {for every sheaf $F\in \Fun^{desc}(\Mf^{op},\bC)$ we have a canonical morphism $\bs(F) \to \cH (F)$.}
\begin{prop}\label{propcom}
{We a}ssume that the $(\infty,1)$-category $\bC$ is stable {and consider a  sheaf
 $F\in \Fun^{desc}(\Mf^{op},\bC)$}. 
 \begin{enumerate}
 \item {The presheaf $\bs(F)$} satisfies descent for finite open covers of a manifold $M$. 
 \item {If the manifold $M$ is compact,} then  the canonical morphism
$\bs(F)(M) \to \cH(F)(M)$ is an equivalence. 
\end{enumerate}
\end{prop}
\begin{proof}
We will use the fact that in a stable and presentable $(\infty,1)$-category
arbitrary colimits commute with finite limits. {This is because the colimit functor
$\Fun(K,\bC) \to \bC$ for every (small) index-category $K$ is a functor between stable $(\infty,1)$-categories. Thus the fact that it preserves colimits implies that it also preserves finite limits.}

For a finite open covering $\{U_i\}_{i \in I}$ of $M$ the limit $\lim_{\Delta} F(U^\bullet)$ can be  rewritten as a limit over a finite subcategory of $\Delta$. 
More precisely we consider the inclusion $j: \Delta^{inj} \to \Delta$ of the subcategory of $\Delta$ with injective maps. Then we consider the diagram
$F(U^\bullet)^{red}: \Delta^{inj} \to \bC$ given by 
$$ F(U^\bullet)^{red}_n = \prod F(U_{i_1} \cap U_{i_2} \cap ... \cap U_{i_n})$$
where the $i_k$ are pairwise different. Now we observe that the original diagram $F(U^\bullet)$ is the right Kan extension of this reduced diagram along the inclusion $j$ which can be shown using the pointwise formulas for right Kan extensions. Thus we conclude that the limit of the reduced diagram is the same as the limit over the original diagram. Finally we note that the reduced diagram has the property that for almost all indices $n$ the object $F(U^\bullet)^{red}_n$ is terminal. Therefore the limit over this diagram is the same as the limit over the finite subcategory of $\Delta$ where $F(U^\bullet)^{red}$ is not terminal.

We now compute
\begin{eqnarray*}
\lim_{i \in \Delta} \bs(F)(U^{i}) &\simeq& \lim_{i \in \Delta} \colim_{j \in \Delta^{op}} F(U^{i} \times \Delta^j) \simeq \colim_{j \in \Delta^{op}} \lim_{i \in \Delta}  F(U^{i} \times \Delta^j) \\&\simeq& \colim_{j \in \Delta^{op}}  F(M\times \Delta^{j})\simeq \bs(F)(M)\ .
\end{eqnarray*}
This shows {the first assertion}.

The evaluation of the sheafification morphism $\bs \to {L\circ \bs} {\simeq}  \cH$ on a manifold $M$ can be described as an {iteration of colimits of limits}  of the form 
$$ {\colim_{\cU}\lim_{\Delta} ~\bs(F)}(U^\bullet) $$
where the colimit {runs over the category whose objects are open covers $\cU $ of $M$ and morphisms are refinements.} For a compact $M$ every open cover can be refined by a finite open cover, i.e. the finite open {covers} are
cofinal among all {open} covers. Thus we can reduce the colimit to a colimit over {the full subcategory of all} finite open covers.   {Given assertion $1.$ the restriction of the system to this subcategory is equivalent to the constant system with value}
$\bs(F)(M)$. This shows the second assertion.
\end{proof}

\begin{lem}\label{thop} For every $\hat E\in \Fun^{desc}(\Mf^{op},\bC)$ and compact manifold $S$ we have a natural equivalence
$$U(I_{S}(\hat E))\simeq U(\hat E)^{S_{top}}\ .$$ 
The equivalence also holds, if $S$ is only homotopy equivalent
to a compact manifold{, i.e. if $S$ has the homotopy type of a finite CW-complex}.
\end{lem}
\begin{proof}
We apply the formula for the homotopification and get:
$$ U(I_{S}(\hat E)) \simeq \colim_{\Delta} (I_S(\hat E)(\Delta^\bullet) )\simeq \colim_\Delta {\hat E}(\Delta^\bullet \times S) \simeq \bs(\hat E)(S) \simeq \cH(\hat E)(S)$$
For the last equivalence we use that $S$ is compact and Proposition \ref{propcom}. Finally we use the general formula
$ \cH(\hat E)(S) \simeq  U(\hat E)^{S_{top}}$ 
which follows from the fact that homotopy invariant sheaves are constant as shown in Proposition \ref{propo}.
\medskip

The second assertion follows from the fact that the functor 
$$U\big(I_{(-)}(\hat E)\big): \Mfo \to \bC$$
is homotopy invariant, which can be seen as follows. First we observe that it suffices to check that it maps the projection $  \Delta^{1}\times M \to M$ to an equivalence
$$ U(I_{M}(\hat E)) \to U(I_{ \Delta^{1}\times M}(\hat E))$$
But this follows from the formula 
$$U(I_{  \Delta^{1}\times M}(\hat E)) \simeq U\big(I_{\Delta^{1}}(I_M(\hat E))\big) \simeq 
U(I_{M}(\hat E))^{\Delta^{1}_{top}} \simeq U(I_{M}(\hat E))\ .$$
\end{proof}

\begin{lem}\label{thointer}
Let $M$ be a compact manifold and $\bC$ stable. Then the functor $I_M: \Fun^{desc}(\Mf^{op},\bC) \to \Fun^{desc}(\Mf^{op},\bC)$ interchanges with colimits. 
\end{lem}
\begin{proof}
As in the proof of Proposition \ref{propcom} we see that for a diagram of sheaves $F\in \Fun^{{desc}}(\Mf^{op},\bC)^{I}$ and $F :=\colim_{i} F_i $ the evaluation on a compact manifold $ {N}$ is given by the colimits of the evaluations
$$ F(N) \simeq \colim_{i }  F_i(N) .$$
Note that this might be wrong for non-compact $N$  since there is a sheafification involved. For compact $N$ we compute:
$$ I_M F(N) = F(M \times N) \simeq \colim_{i }  F_i(N \times M)  \simeq \colim_{i  } {\big(}I_M F_i(N){\big)}   \simeq \big(\colim_{i  }(I_M F_i)\big)(N)
$$
This shows  that the canonical morphism $\colim_{i}  I_M F_i \to I_{M} F$  becomes an equivalence  after evaluation at arbitrary compact manifolds $N$, in particular for all spheres. We finally apply  Lemma \ref{test_equi}. 
\end{proof} 
 
Let $\iota:\Ch\to \Ch[W^{-1}]$ be the localization of the category of chain complexes 
at the quasi-isomorphisms. Then $\Ch[W^{-1}]$ is a stable and presentable $(\infty,1)$-category.
At several {occasions} we need the following fact:
\begin{lem}\label{heut1}
The localization $\iota:\Ch\to \Ch[W^{-1}]$ preserves filtered colimits.
\end{lem}
\proof
Let $I$ be a filtered category and $A\in \Fun(I,\Ch)$.
We use the fact that there exists a model category structure on $\Ch$ which can be used
in order to construct  models of $\Ch[W^{-1}]$ and $\Fun(I,\Ch[W^{-1}])$, see  \cite{MR1650134}. 
In particular, if $A\to Q(A)$ is a cofibrant replacement of $A$, then  
$\colim_{I}  Q(A)$ is a model for $ \colim_{I} \iota(A)$. Since cohomology commutes with filtered colimits
the quasi-isomorphism
$A\to Q(A)$ induces a weak equivalence $\colim_{I}A\to \colim_{I}Q(A)$.
\hB

For an abelian group $U$ we let $U[-n]$ denote the chain complex obtained by placing $U$ in degree $n$, and we use the same convention for sheaves of abelian groups.
{The following result is well known, a proof in this form can e.g. be found in 
in \cite[Problem 4.3.2]{skript}}.
\begin{lem} \label{lll23} For a cosimplicial chain complex  $A\in \Fun(\Delta,\Ch)$ we have an equivalence 
 $$\lim_{\Delta} \iota(A){\simeq}\iota(\tot(A))$$
\end{lem}
Here the total complex  $\tot(A)$ of a cosimplicial chain complex  $A$  {is} given by 
$$\tot(A)^{n}:=\prod_{p+q=n} A^{p}([q])$$ with the differential
$$d (x)=(-1)^{q}d^{A}x+\delta x\ , \quad  \delta(x):=\sum_{i=0}^{q} (-1)^{i}\partial_{i} x\ , \quad x\in A^{p}([q])\ .$$

 A proof of the following Lemma can be found in \cite[Lemma 4.24]{skript}. 
 \begin{lem}\label{pi12} For a simplicial complex $A\in \Fun(\Delta^{op},\Ch)$ we have an equivalence 
 $$\colim_{\Delta^{op}} \iota(A){\simeq} \iota(\tot(A))\ .$$
\end{lem}

Here $\tot(A)$ of a simplicial chain complex is given by 
$$\tot(A)^{n}:=\bigoplus_{p-q=n} A^{p}([q])$$ with the differential
$$d (x)=(-1)^{q}d^{A}x+\delta x\ , \quad  \delta(x):=\sum_{i=0}^{q} (-1)^{i}\partial_{i} x\ , \quad x\in A^{p}([q])\ .$$
The category $\Mf$ has a structure sheaf $\cC^{\infty}$ of rings of smooth functions. We can consider sheaves of $\cC^{\infty}$-modules. 
\begin{lem}\label{hj12}
If $A\in \Fun^{desc}(\Mf^{op},\Ch)$ is a complex of sheaves, whose components have the structure of
$C^{\infty}$-modules, then $\iota(A)$ is also a sheaf, i.e. $\iota(A) \in \Fun^{desc}(\Mf^{op},\Ch[W^{-1}])$.
\end{lem}
\proof
Let $U\to M$ be a covering. We must show that
$\iota(A)(M)\to \lim_{\Delta} \iota(A)(U^{\bullet})$ is an equivalence. In view of Lemma \ref{lll23} it suffices to show that $\iota(A)(M)\to  \iota(\tot(A(U^{\bullet})))$ is an equivalence.
We now use the existence of smooth partitions of unity in order to show that $\tot(A)(U^{\bullet})$ is acyclic in the $\delta$-direction. \hB

\begin{lem}\label{zzz2}
 For sheaf $A\in \Fun^{desc}(\Mf^{op},\Ab)$ of $\cC^{\infty}$-modules we have an equivalence $\cH(\iota(A[0])) {\simeq}  0$.
 \end{lem}
 \proof
 
 Using Lemma \ref{pi12} we get an equivalence
 $$\bs \circ \iota(A[0])(M) \simeq \iota\big(  ...   \to A(\Delta^2{\times M}) \to A(\Delta^{1}{\times M}) \to A(\Delta^0{\times M}) \to 0 \to ...\big).$$
 Since $A$ is a $\cC^\infty$-module this complex is acyclic. Thus with Lemma \ref{lll25} we get {the equivalences}
  $\cH(\iota(A[0])) \simeq L \circ \bs(\iota(A[0])) \simeq 0$.
  \hB 

\begin{ddd}\label{heute10} For presheaf of chain complexes $F\in\Fun (\Mf^{op},\Ch)$ we define the stupid truncations
$$\sigma^{\ge m}F\::\: \dots\to 0\to F^{m}\to F^{m+1}\to \dots$$
and
$$\sigma^{\le m-1}F\::\: \dots F^{m-2}\to F^{m-1}\to 0\to \dots\ .$$
\end{ddd} 

\bigskip

 We consider a  sheaf of chain complexes $F\in \Fun^{desc}(\Mf^{op},\Ch)$   whose components are sheaves of $\cC^{\infty}$-modules. Note that by Lemma \ref{hj12} we have
 $\iota (F)\in \Fun^{desc}(\Mf^{op},\Ch[W^{-1}])$
and for   every $k\in \Z$ that
 $\iota(\sigma^{\ge -k} (F))  \in \Fun^{desc}(\Mf^{op},\Ch[W^{-1}])$.
  \begin{lem}\label{uuzzii1} Let $F\in \Fun^{desc}(\Mf^{op},\Ch)$  be a  sheaf of chain complexes   whose components are sheaves of $\cC^{\infty}$-modules. For every $k\in \Z$ the natural map induces an equivalence
 $$\cH(\iota(\sigma^{\ge -k} (F)))\to   \cH(\iota(F))\ .$$
 \end{lem}
 \begin{proof}
For every sheaf of chain complexes $F\in \Fun^{desc}(\Mf^{op},\Ch)$ we have an equivalence
$$ \colim_{n\in \Z} \sigma^{\ge -n} (F)\stackrel{\sim}{\to} F$$
induced by the natural inclusions. Since   $\iota$ commutes with filtered colimits 
it induces an equivalence
$$  \colim_{n\in \nat} \iota(\sigma^{\ge -n} (F))\stackrel{\sim}{\to} \iota(F)\ .$$  
The exact sequences
$$0\to \sigma^{\ge -n+1}(F)\to \sigma^{\ge -n} (F)\to F^{-n}[n]\to 0$$
induce fibre sequences after application of $\cH\circ\iota$. Using that
 $\cH(\iota(F^{-n}[n])) \simeq  0$ by Lemma \ref{zzz2} we obtain the equivalences
$$\cH(\iota(\sigma^{\ge -n-1}(F)))\stackrel{\sim}{\to }\cH(\iota(\sigma^{\ge -n}(F)))$$
for all $n\in \Z$. We get a chain of equivalences
\begin{eqnarray*}
   \cH(\iota(F)) \simeq \cH(\colim_{n\in \Z} \iota(\sigma^{\ge -n} (F)))  \simeq  \colim_{n\in \Z} \cH(\iota(\sigma^{\ge -n} (F)))&&\\ \simeq  \colim_{n\in \Z} \cH(\iota(\sigma^{\ge -k} (F))) \simeq  \cH(\iota(\sigma^{\ge -k}(F)))&& \ .
 \end{eqnarray*} 
 \end{proof}

   \begin{kor}\label{rrr4} Let $F\in \Fun^{desc}(\Mf^{op},\Ch)$  be a  sheaf of chain complexes   whose components are sheaves of $\cC^{\infty}$-modules. For every $k\in \Z$ we have
   $$\cH(\iota(\sigma^{\le -k}(F)){)\simeq} 0\ .$$
   \end{kor}
   \proof
 We have an exact sequence
  $$0\to \sigma^{\ge -k+1} (F)\to F\to \sigma^{\le -k}(F)\to 0$$
  which induces a fibre sequence after application of $\cH\circ \iota$.
  We now use that
  $$\cH(\iota( \sigma^{\ge -k+1} (F)))\to \cH(\iota(F))$$ is an equivalence by Lemma \ref{uuzzii1}
  in order to conclude the assertion.
    \hB
    
Let $\rho:A\to B$ be a homomorphism of abelian group. Then $A$ acts on $B$ by
$a\cdot b:=b+\rho(a)$. The abelian group structures  on $A$ and $B$ turn action groupoid
$A\times B{\rightrightarrows} B$ into a symmetric monoidal groupoid, i.e. an object of
$\CommMon(\Groupoids)$. Since this is a Picard groupoid we have
$\iota(\Nerve(A\times B{\rightrightarrows} B))\in \Comm\Grp(\sSet[W^{-1}])$. 
In the  following Lemma we also consider $A\to B$ as a chain complex, where $B$ sits in degree zero.
Recall the functor $\sp$ from \eqref{wie3}.

 \begin{lem}\label{rrr6}
  There is an natural equivalence  
$$\sp(\Nerve(\iota(B\times A{\rightrightarrows} B)))  \simeq H(\iota( A  \to B  )) $$
in $\Sp$. 
\end{lem}

\proof 
We have pull-back diagrams in $\CommMon(\Groupoids[W^{-1}])$ and
$\Ch[W^{-1}]$, respectively:
$$\xymatrix{\iota(A\times B{\rightrightarrows} B)\ar[d]\ar[r]&{*}\ar[d]\\
\iota(A {\rightrightarrows} *)\ar[r]^{\rho}&\iota(B{\rightrightarrows} *)
}\ ,\quad \xymatrix{
\iota(A\to B) \ar[d]\ar[r]&{*}\ar[d]\\
\iota (A[1])  \ar[r]^{\rho[1]}&\iota (B[1]) 
}\ .
$$
If we apply $\Nerve$ or $H$, respectively, then we obtain pull-back diagrams in
$\Comm\Grp(\sSet[W^{-1}])$ or $\Sp$, respectively:
$$\xymatrix{\Nerve(\iota(A\times B\Rightarrow B))\ar[d]\ar[r]&{*}\ar[d]\\
\Nerve(\iota(A\Rightarrow *))\ar[r]&\Nerve(\iota(B\Rightarrow *))}\ ,\quad \xymatrix{H(\iota(A\to B)) \ar[d]\ar[r]&{*}\ar[d]\\
H(\iota (A[1]))  \ar[r]&H(\iota (B[1])) }\ .$$
For the right square we also use that nerves of Picard stacks are grouplike spaces, and that a pull-back of grouplike spaces is detected by the diagram of underlying spaces.
Using the equivalences (see \eqref{pqwe1})
$$\sp(\Nerve(\iota(A{\rightrightarrows} *))) \simeq H(\iota (A[1]))     \ ,\quad \sp(\Nerve(\iota(B{\rightrightarrows} *)))
 \simeq H(\iota (B[1]))  $$
and the fact that the right square is actually pull-back in connective spectra (since again this is detected in spectra) we
conclude the equivalence
$$   \sp(\Nerve(\iota(A\times B{\rightrightarrows} B)))  \simeq H(\iota(A\to B))\ .$$ \hB
\bibliographystyle{alpha}
\bibliography{proje}
\end{document}